\documentclass[12pt]{article}

\usepackage{geometry}
\usepackage[english]{babel}
\usepackage{amsmath}
\usepackage{amssymb}
\usepackage{amsthm}
\usepackage{libertine}
\usepackage[libertine]{newtxmath}
\usepackage[T1]{fontenc}
\usepackage[utf8]{inputenc}
\usepackage{mathtools}
\usepackage{graphicx}
\usepackage{enumitem}
\usepackage{hyperref}
\usepackage{mathrsfs}
\usepackage{courier}
\usepackage{color}
\usepackage{tikz}
\usepackage{subfig}

\DeclarePairedDelimiter{\abs}{\lvert}{\rvert}
\DeclarePairedDelimiter{\norma}{\lVert}{\rVert}

\DeclareMathOperator{\supp}{supp}

\DeclareMathOperator{\Div}{div}

\newcommand{\numberset}{\mathbb}
\newcommand{\N}{\numberset{N}}

\newcommand{\Z}{\numberset{Z}}
\newcommand{\R}{\numberset{R}}
\newcommand{\D}{\mathscr{D}}
\newcommand{\E}{\mathcal{E}}

\newcommand{\diff}{\mathrm{d}}

\newcommand{\p}{\mathcal{P}}
\newcommand{\eps}{\epsilon}

\newtheorem{theorem}{Theorem}
\newtheorem{lemma}{Lemma}
\newtheorem{proposition}{Proposition}
\newtheorem{definition}{Definition}
\theoremstyle{plain}
\newtheorem{corollary}{Corollary}
\newtheorem{remark}{Remark}

\numberwithin{equation}{section}

\newcommand{\mybinom}[2]{\biggl(\genfrac{}{}{0pt}{}{#1}{#2}\biggr)}

\begin{document}

\title{On a Babu\v{s}ka paradox for polyharmonic operators: spectral stability and boundary homogenization for intermediate problems}

\author{Francesco Ferraresso\thanks{Institute of Mathematics, Universit\"at Bern, Sidlerstrasse 5, 3012 Bern, Switzerland, \texttt{francesco.ferraresso@math.unibe.ch}} \: and \: Pier Domenico Lamberti\thanks{Dipartimento di Matematica ``Tullio Levi-Civita'', Universit\`a degli Studi di Padova, Via Trieste 63, 35121 Padova, Italy, \texttt{lamberti@math.unipd.it}}}

\maketitle

\begin{abstract}
We analyse the spectral convergence  of high order elliptic  differential operators  subject to singular domain perturbations and homogeneous boundary conditions of intermediate type.
We identify sharp assumptions on the domain perturbations improving, in the case of polyharmonic operators of higher order, conditions known to be sharp in the case of fourth order operators. The optimality is proved by analysing in detail a boundary homogenization problem, which provides a smooth version of a polyharmonic Babu\v{s}ka paradox.
\end{abstract}



\section{Introduction}
A recurrent topic  in the Analysis of Partial Differential Equations, in Spectral Theory, and their applications is the study of the variation of the solutions to elliptic boundary value problems on domains subject to boundary perturbation, with contributions rooting back in the works of Courant and Hilbert \cite{C-H}, and Keldysh \cite{Kel}. 
The mathematical  interest in this type of  problems is also given by the possible appearance of an unexpected asymptotic behaviour of the solutions, which can be understood as a spectral instability phenomenon. 
Probably the most famous example in elasticity theory  is the celebrated Babu\v{s}ka  paradox which concerns the approximation of a thin hinged circular plate by means of an invading sequence of convex polygons. 
This problem was considered by Babu\v{s}ka in \cite{Bab}  and was further discussed  by Maz'ya and Nazarov 
in \cite{MazNaz} where among various results they present a variant of the Babu\v{s}ka paradox consisting in the approximation a thin hinged circular plate by means of an invading sequence of non-convex, indented polygons  (see \cite[\S~1.4]{GazzGS} for a recent discussion on this subject and for more details concerning the related results of Sapond\v{z}hyan~\cite{Sap}).  We find convenient to briefly recall the formulation of the paradox.   

Given a circle $\Omega$ in $\R^2$ and a datum  $f\in  L^2(\Omega )$, consider the following boundary value problem 
\begin{equation}
\label{intro: convex}
\begin{cases}
\Delta^2 u = f, \quad &\textup{in $\Omega$},\\
u=0, \quad &\textup{on $\partial \Omega$}, \\
\frac{\partial^2 u}{\partial n^2}= 0, \quad &\textup{on $\partial \Omega$,}
\end{cases}
\end{equation}
in the unknown real-valued function $u$. Note that here and in the sequel, boundary value problems will be understood in the weak sense. Thus, problem 
\eqref{intro: convex} consists in finding $u\in W^{2,2}(\Omega)\cap W^{1,2}_0(\Omega) $ such that 
\begin{equation*}
\int_{\Omega}D^2u:D^2\varphi \, dx = \int_{\Omega }f\varphi\, dx,\ \ {\rm for\ all}\ \varphi\in  W^{2,2}(\Omega)\cap W^{1,2}_0(\Omega),
\end{equation*}
where $  D^2u:D^2\varphi =\sum_{i,j=1}^Nu_{x_ix_j}\varphi_{x_ix_j}$ is the Frobenius product of the two Hessian matrices of $u$ and $\varphi$.
In the theory of elastic plates, $u$ represents the deflection of a {\it hinged} thin plate with midplane $\Omega$  and  normal load $f$.  

Define  inside $\Omega$ an invading  sequence of indented polygons $\Omega_n$ obtained by modifying an inscribed  convex polygon with $n$ vertexes $p_j^n$, $j=1, \dots ,n$,  and replacing its contour line in a neighbourhood of each  $p_j^n$ by  a  $V$-shaped line as in Figure[1]. The small curvilinear triangles appearing have height equal to $h_j^n$ and base of length $\eta_j^n$, while   the length of the nearby chord   (the side of the  polygon) is denoted by  $\zeta_j^n$.  
Consider now the same boundary value problem in $\Omega_n$
\begin{equation}
\label{intro: polygon}
\begin{cases}
\Delta^2 u_n = f, \quad &\textup{in $\Omega_n$},\\
u_n=0, \quad &\textup{on $\partial \Omega_n $}, \\
\frac{\partial^2 u_n}{\partial n^2}= 0, \quad &\textup{on $\partial \Omega_n$,} \\
\end{cases}
\end{equation}
in the unknown  $u_n\in W^{2,2}(\Omega_n)\cap W^{1,2}_0(\Omega_n)$.
The paradox lies in the fact that if
\[
\max_{1 \leq j \leq n} \frac{|\zeta_j^n|}{|\eta_j^n|} = O(1), \quad \max_{1 \leq j \leq n} \frac{|\eta_j^n|}{|h_j^n|^{2/3}} = o(1),
\]
as $n \to \infty$, then  the solution $u_n \in W^{2,2}(\Omega_n) \cap W^{1,2}_0(\Omega_n)$ of \eqref{intro: polygon} does not converge to the solution $u$ of \eqref{intro: convex}, but to  the solution $v$ of the boundary value problem 
\begin{equation}
\label{intro: convex2}
\begin{cases}
\Delta^2 v = f, \quad &\textup{in $\Omega$},\\
v=0, \quad &\textup{on $\partial \Omega$}, \\
\frac{\partial v}{\partial n} = 0, \quad &\textup{on $\partial \Omega$}. \\
\end{cases}
\end{equation}
Here $v$ represents the deflection of a {\it clamped } thin plate.
 Note that  it is possible to choose $|\zeta_j^n|=0$ for all $j$ and $n$ in order to obtain the wild looking set $\Omega_n$ in Figure 2.

\begin{figure}
    \centering
    \begin{minipage}[b]{0.49\textwidth}
        \centering
         \begin{tikzpicture}[scale=0.9]
			\draw[thick] (0,0) circle (3cm);
			\draw[very thick] ({3*cos(27.5)}, {-3*sin(27.5)}) -- ({3*cos(2.5)}, {-3*sin(2.5)}) -- (2.5, 0) -- ({3*cos(2.5)},{3*sin(2.5)}) 				-- ({3*cos(27.5)},{3*sin(27.5)}) -- ({2.5*cos(30)},{2.5*sin(30)}) -- ({3*cos(32.5)},{3*sin(32.5)}) -- ({3*cos(57.5)},						{3*sin(57.5)}) -- ({2.5*cos(60)},{2.5*sin(60)}) -- ({3*cos(62.5)},{3*sin(62.5)});
			\draw[very thick, blue] ({3*cos(62.5)},{3*sin(62.5)}) -- ({3*cos(87.5)},{3*sin(87.5)});
			\draw[very thick] ({3*cos(87.5)},{3*sin(87.5)}) -- (0,2.5) -- ({3*cos(92.5)},{3*sin(92.5)}) -- ({3*cos(117.5)},{3*sin(117.5)}) 			-- ({2.5*cos(120)},{2.5*sin(120)})--({3*cos(122.5)},{3*sin(122.5)}) -- ({3*cos(147.5)},{3*sin(147.5)}) -- ({2.5*cos(150)},					{2.5*sin(150)}) -- ({3*cos(152.5)},{3*sin(152.5)}) -- ({3*cos(177.5)},{3*sin(177.5)}) -- (-2.5,0) -- ({-3*cos(2.5)},						{-3*sin(2.5)}) -- ({-3*cos(27.5)},{-3*sin(27.5)}) -- ({-2.5*cos(30)},{-2.5*sin(30)}) -- ({-3*cos(32.5)},{-3*sin(32.5)}) -- 					({-3*cos(57.5)},{-3*sin(57.5)}) -- ({-2.5*cos(60)},{-2.5*sin(60)}) -- ({-3*cos(62.5)},{-3*sin(62.5)}) -- ({-3*cos(87.5)},					{-3*sin(87.5)}) -- (0,-2.5) -- ({-3*cos(92.5)},{-3*sin(92.5)}) -- ({-3*cos(117.5)},{-3*sin(117.5)}) -- ({-2.5*cos(120)},					{-2.5*sin(120)})--({-3*cos(122.5)},{-3*sin(122.5)}) -- ({-3*cos(147.5)},{-3*sin(147.5)}) -- ({-2.5*cos(150)},{-2.5*sin(150)}) 				-- ({3*cos(27.5)}, {-3*sin(27.5)}) ;
			\draw[thick, red] ({3*cos(87.5)},{3*sin(87.5)}) arc(87.5: 92.5 : 3);
			\draw ({2.9*cos(75)}, {2.9*sin(75)}) -- ({3.5*cos(75)},{3.5*sin(75)}) ;
			\draw ({3.7*cos(73)},{3.7*sin(73)}) node{{\tiny \color{blue} $\zeta^{n}_j$}};
			\draw ({0},{3}) -- ({0},{3.5});
			\draw (0, 3.7) node{{\tiny \color{red} $\eta^{n}_j$}};
			\draw[thick, orange] (0, 2.51) -- (0,2.98);
			\draw (0,2.75) -- (-0.6, 2);
			\draw (-0.7, 1.8) node{{\tiny \color{orange} $h_j^{n}$}};
			\end{tikzpicture}
        	\caption{Indented polygon}
    	\end{minipage}\hfill
    \begin{minipage}[b]{0.49\textwidth}
        \centering
        \includegraphics[width=0.8\textwidth]{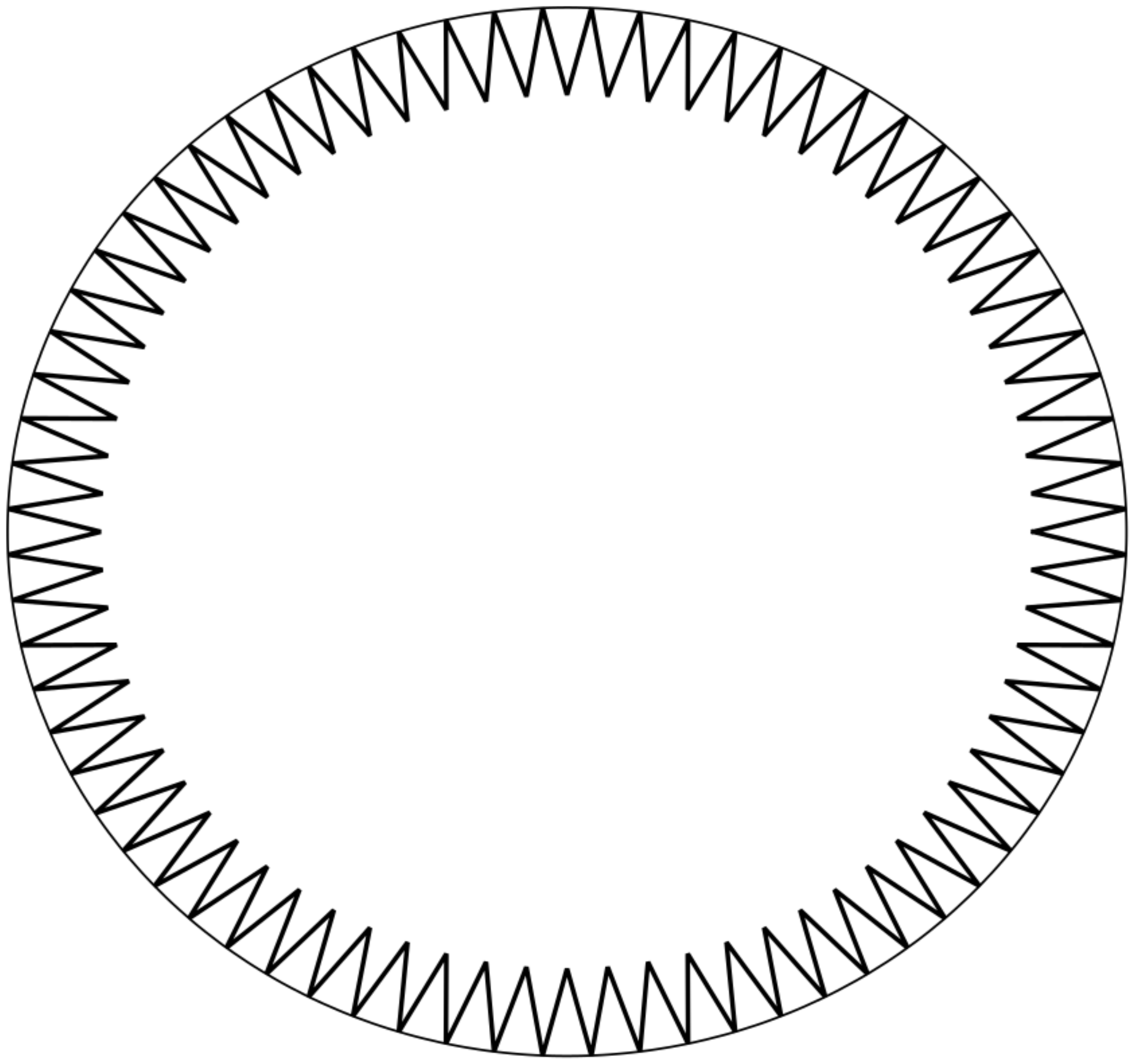} 
       \caption{Degenerate indented polygon}
    \end{minipage}
\end{figure}

In \cite{arlacras, ArrLamb} the authors considered a smooth version of this paradox. Given a sufficiently regular bounded domain $W$ in $\R^{N-1}$, $N\geq 2$,  they define a family of domains $(\Omega_\eps)_{0<\eps<\eps_0}$ by setting
\begin{equation}
\label{intro: geometry}
\Omega = W \times (-1,0), \,\,\, \Omega_\eps = \{ (\bar{x}, x_N) \in \R^N : \bar{x} \in W, -1< x_N < \eps^\alpha b(\bar{x}/\eps)\}
\end{equation}
where $\bar{x}= (x_1, \dots, x_{N-1})$, and $b$ is a non-constant, smooth, positive, periodic function of period $Y = [-1/2, 1/2]^{N-1}$. The geometry of this perturbation is described in figure [3] below.

By comparing Figure \ref{fig: homogenization}(a) and Figure 2, one realizes that  the perturbations look similar  locally at the boundary. This analogy goes further if we define $h_j^{n} = \eps^{\alpha}$ and $\eta_j^{n} = \eps$, with $\epsilon =1/n$. Indeed, in \cite{ArrLamb} it was proved that if 
\[
\frac{|\eta_j^{n}|}{|h_j^{n}|^{2/3}} = \frac{\eps}{\eps^{2/3 \alpha}} = o(1),
\]
as $\eps \to 0$, that is if $\alpha < 3/2$, then the same  Babu\v{s}ka-type paradox  appears. Moreover,  it was also proved that if $\alpha > 3/2$ then no Babu\v{s}ka paradox appears and there is spectral stability. The threshold $\alpha = 3/2$ is then critical and represents a typical case of study for homogenization theory: in fact, it was proved in \cite{ArrLamb} that the limiting problem contains a `strange term' which could be interpreted as a `strange curvature'.

\begin{figure}
\centering
\subfloat[$\alpha = 1$]{\includegraphics[width=.4\textwidth]{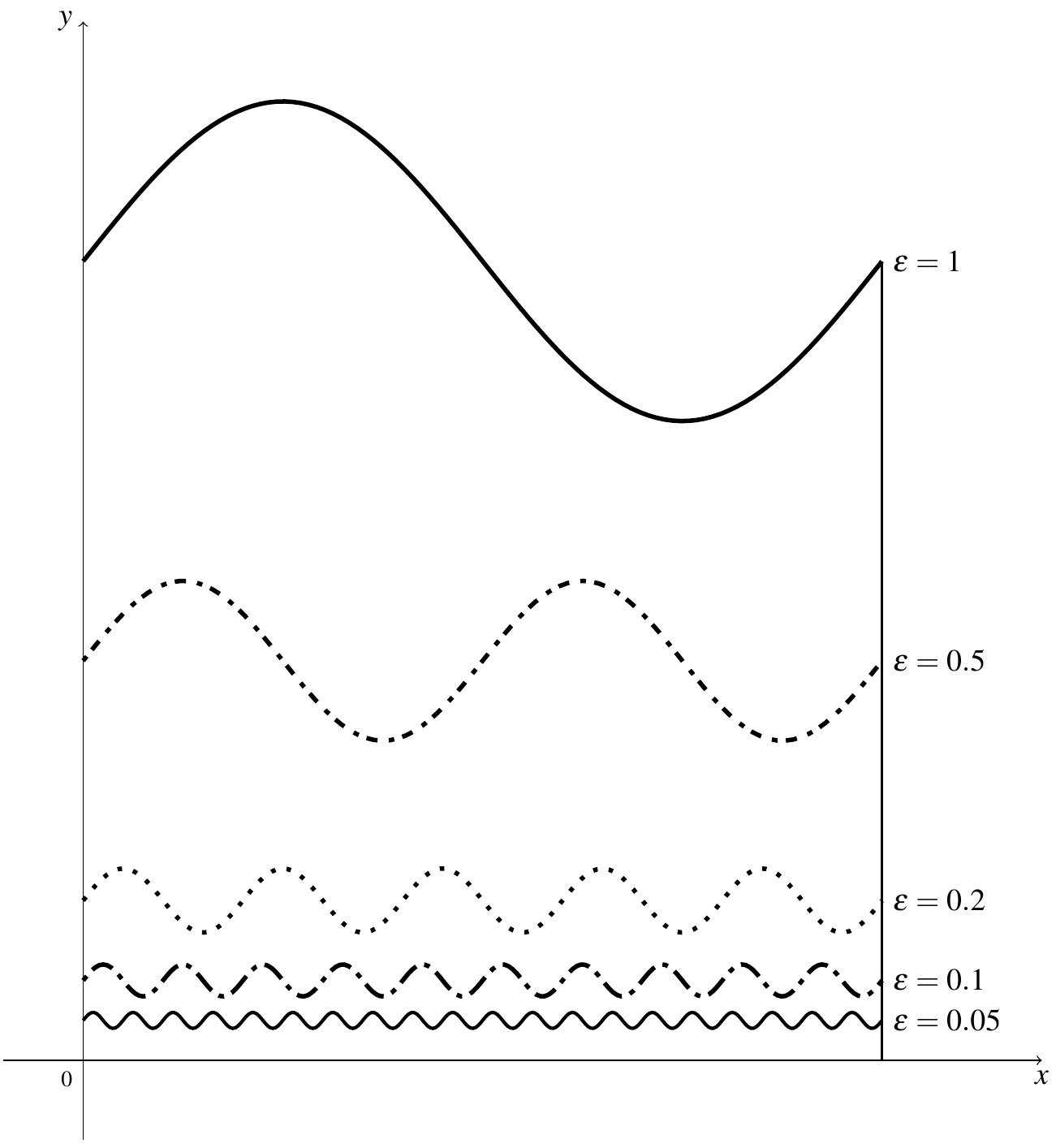}}
\subfloat[$\alpha = 2$]{\includegraphics[width=.4\textwidth]{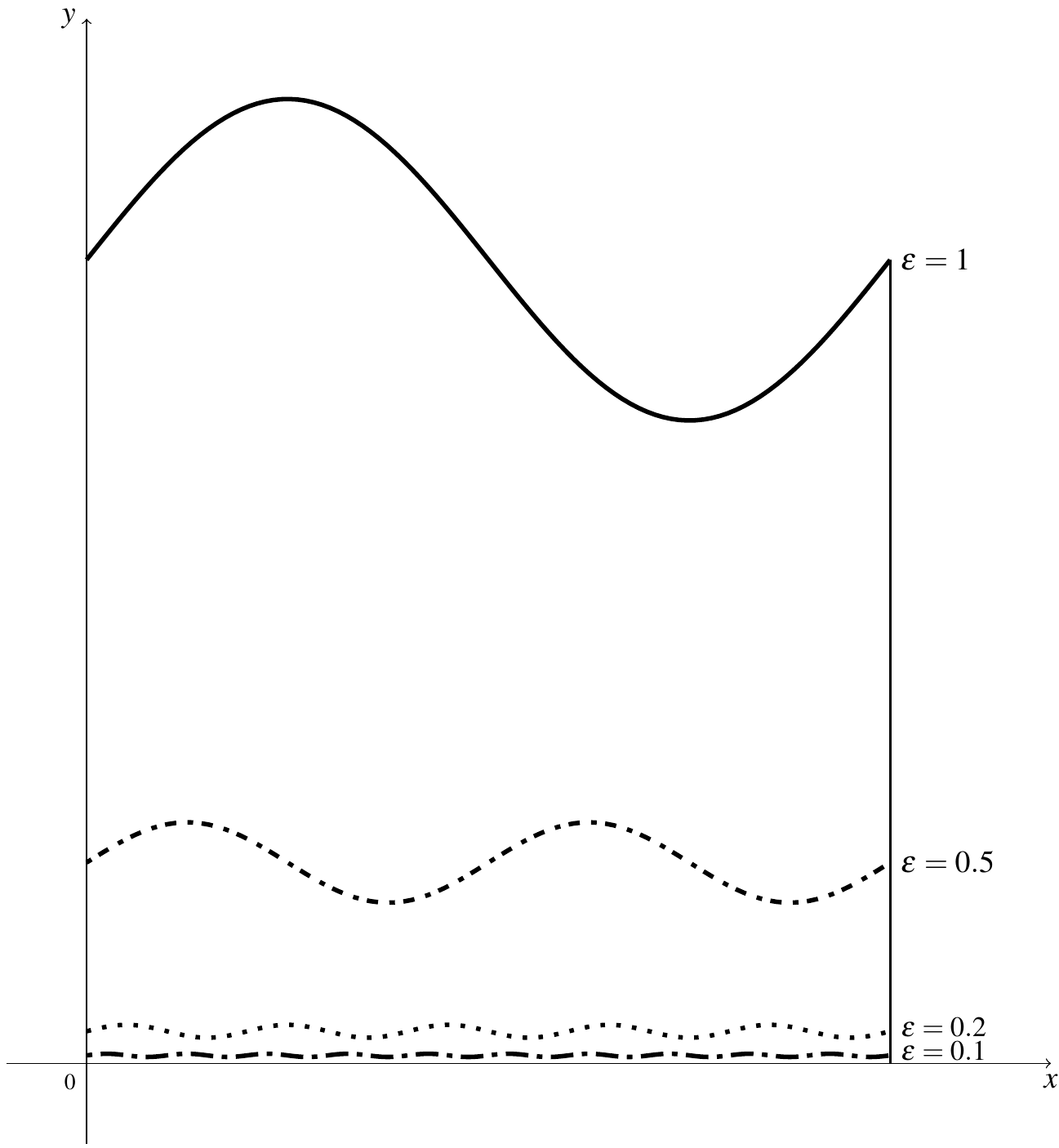}}
\caption{\small Oscillations of the upper boundary of $\Omega_\eps$ as $\eps \to 0$, depending on $\alpha$.}
\label{fig: homogenization}
\end{figure}

It is then natural to wonder whether Babu\v{s}ka-type paradoxes  may be detected in the case of polyharmonic operators $(-\Delta)^m$, $m>2$ subject to intermediate boundary conditions. The answer is not as straightforward as it may appear, and it is necessary to clarify first what are the possible boundary conditions for  those operators. Indeed, there exists a whole family of boundary value problems depending on a parameter $k=0,1\dots , m$, the weak formulation of which reads as follows: given a bounded domain (i.e., a connected open set) $\Omega$ in $\R^N$ with sufficiently smooth boundary, $m\in \N$, and $f\in L^2(\Omega)$, find $u\in W^{m,2}(\Omega)\cap W^{k,2}_0(\Omega)$ such that
\begin{equation}
\label{weakmain}
\int_{\Omega}D^mu:D^m\varphi dx+\int_{\Omega}u\varphi dx= \int_{\Omega }f\varphi dx,\ \ \ \forall \varphi\in  W^{m,2}(\Omega)\cap W^{k,2}_0(\Omega).
\end{equation}
Here we denote by  $W^{m,2}(\Omega)$ the standard Sobolev space of functions in $L^2(\Omega)$ with weak derivatives up to order $m$ in $L^2(\Omega)$ and by  $W^{k,2}_0(\Omega)$  the closure in $W^{k,2}(\Omega)$ of the $C^{\infty}$-functions with compact support in $\Omega$.
Note that for $k=m$ one obtains the Dirichlet problem 
\begin{equation}
\label{classicalmaindir}
\begin{cases}
(-\Delta)^m u + u = f, &\textup{in $\Omega$,}\\
\frac{\partial^l u}{\partial n^l} = 0, &\textup{on $\partial \Omega$,\quad\quad for all $0 \leq l \leq m-1$,}\\
\end{cases}
\end{equation}
while for $k=m-1$ one gets the significantly different problem
\begin{equation}
\label{classicalmain}
\begin{cases}
(-\Delta)^m u + u = f, &\textup{in $\Omega$,}\\
\frac{\partial^l u}{\partial n^l} = 0, &\textup{on $\partial \Omega$,\quad\quad for all $0 \leq l \leq m-2$.}\\
\frac{\partial^m u}{\partial n^m} = 0, &\textup{on $\partial \Omega$.}
\end{cases}
\end{equation}
Finally, for $k=0$ one gets the problem with natural boundary conditions, also known as Neumann problem, and this explains why problem  \eqref{classicalmain} is called intermediate.  Actually, in this paper we refer to problem  \eqref{classicalmain} as to the {\it strong intermediate problem} to emphasise the fact that   \eqref{classicalmain}  is the intermediate problem with the largest $k$  and to distinguish it from the other cases where $0<k<m-1$ which are called here {\it weak intermediate problems}.   According to these considerations, one is led to ask the following:\\

\noindent {\bf Question}: {\it Are there Babu\v{s}ka-type paradoxes for polyharmonic operators $(-\Delta)^m$, $m>2$ satisfying  intermediate boundary conditions, and which are the natural assumptions which prevent the appearance of this paradox? } \\

We are able to  answer to this question in the geometric setting  given by  \eqref{intro: geometry}. Since when $m=2$ problem \eqref{classicalmain} coincides with the hinged plate  \eqref{intro: convex}, the Babu\v{s}ka paradox  will be discussed for polyharmonic operators with strong intermediate boundary conditions (in short, $(SIBC)$), being the natural higher order version  of the  intermediate boundary conditions for the biharmonic operator. 

 Let us describe one of the two main results of this paper. 
Let $\Omega_\eps$ and $\Omega$ be as in \eqref{intro: geometry}, $V(\Omega_\eps) = W^{m,2}(\Omega_\eps) \cap W^{m-1,2}_0(\Omega_\eps)$. For every $\eps > 0$, let $u_\eps \in V(\Omega_\eps)$ be the solution of 
\begin{equation}\label{intro: mainweak}
\int_{\Omega_\eps} D^m\! u_{\eps} : D^m\! \varphi + u_{\eps}\varphi \, dx = \int_{\Omega_\eps} f \varphi\, dx,\ \ \ {\rm for \ all }\ \ \varphi \in V(\Omega_\eps).
\end{equation}
Recall that this is the weak formulation of the Poisson problem for $(-\Delta)^m + \mathbb{I}$ with $(SIBC)$. For $u \in W^{m,2}(\Omega)$, define $T_\eps u = u \circ \Phi_\eps$ where $\Phi_\eps$ is a smooth diffeomorphism mapping $\Omega_\eps$ into $\Omega$ that coincides with the identity on a large part $K_\eps$ of $\Omega$, with $|\Omega \setminus K_\eps| \to 0$ as $\eps \to 0$, see \eqref{def: T eps high}. Let $u$ be such that $\norma{u_\eps - T_\eps u}_{L^2(\Omega_\eps)} \to 0$ as $\eps \to 0$. 

 Theorem \ref{thm: poly strong} states that the limit $u$ solves different differential problems according to the values of the parameter $\alpha$. More precisely,  we have the following trichotomy:
\begin{enumerate}[label=(\roman*)]
\item {\it (Stability)}  If $\alpha > 3/2$, then $u$ solves \eqref{intro: mainweak} in $\Omega$, that is, $u$ satisfies $(-\Delta)^m u + u = f$ in $\Omega$ and $(SIBC)$ on $\partial \Omega$;
\item {\it  (Degeneration)} If $\alpha < 3/2$, then $u$ satisfies $(-\Delta)^m u + u = f$ in $\Omega$, with Dirichlet boundary conditions on $W \times \{0\}$, that is 
$$\frac{\partial^l u}{\partial n^l} = 0,\ \ \textup{for all}\ 0 \leq l \leq m-1,$$
 and $(SIBC)$ on the rest of the boundary of $\Omega$;
\item {\it (Strange term)} If $\alpha = 3/2$, then $u$ satisfies $(-\Delta)^m u + u = f$ in $\Omega$ with  the following boundary conditions on $W \times \{0\}$
$$
\left\{\begin{array}{ll}
D^l u=0,& \textup{ for\  all}\   0\le l \leq m-2,\vspace{1mm}\\
\frac{\partial^m u}{\partial n^m} + K \frac{\partial^{m-1} u}{\partial n^{m-1}} = 0,&
\end{array}
\right.$$
and $(SIBC)$ on the rest of the boundary of $\Omega$. Here  $K$  is a certain positive constant that can be characterized as the energy of a suitable $m$-harmonic function in $Y \times (-\infty,0)$.
\end{enumerate}

It follows that if $\alpha < 3/2$ a polyharmonic Babu\v{s}ka paradox appears. It is interesting to observe that  the critical value $3/2$ is the same for all the polyharmonic operators with $(SIBC)$. 

The techniques used to prove Theorem \ref{thm: poly strong} vary drastically depending on the case $(i)-(iii)$ considered. Theorem \ref{thm: poly strong}(i) is a consequence of Theorem \ref{lemma: 6.2 improved}, which is the second main result of the paper and provides a general stability criterion for self-adjoint elliptic differential operators of order $2m$ with non-constant coefficients and compact resolvents (or, more precisely, for their realization in the space $W^{m,2}(\Omega) \cap W^{k,2}_0(\Omega)$, $0<k<m$) on varying domains featuring a fast oscillating boundary.

Theorem \ref{lemma: 6.2 improved} is an improvement of a previous result (see \cite[Lemma 6.2]{ArrLamb}) and can be summarized and simplified  in the following way. 
Let $\Omega $ and $\Omega_{\epsilon}$ be bounded domains in $\R^N$ defined as follows:
$$\begin{array}{l}
\Omega  = \{(\bar{x}, x_N) \in W \times (a,b) : \bar{x} \in W, a < g(\bar{x}) < b\},\vspace{1mm}\\
\Omega_\eps  = \{(\bar{x}, x_N) \in W \times (a,b) : \bar{x} \in W, a < g_{\eps}(\bar{x}) < b\},
\end{array}
$$
where $W\subset \R^{n-1}$ is as above, $a + \rho < g,g_{\eps} < b - \rho$, $a,b\in \R$, and $g,g_\eps \in C^m(\overline{W})$. If $\norma{g-g_\eps}_{\infty} $  converges to zero as $\eps  $ goes to zero  and, for all $|\beta |=m$, $\norma{D^\beta(g - g_\eps)}_{\infty}$  converges to zero or diverges to infinity with  a suitable rate  expressed in terms of a power of $\norma{g-g_\eps}_{\infty}$, then the spectrum of the realization of a self-adjoint elliptic differential operator in $ W^{m,2}(\Omega_\eps) \cap W^{k,2}_0(\Omega_\eps)$, $1 \leq k \leq m-1$ is stable as $\eps \to 0$ . We note that \cite[Lemma 6.2]{ArrLamb} is sharp in the case $m=2$ and $k=1$. In Theorem \ref{lemma: 6.2 improved} we allow a  rate of convergence or divergence for $\norma{D^\beta(g - g_\eps)}_{\infty}$ which is much better  when $k>1$. For example, going back to Theorem \ref{thm: poly strong}(i), we note the following fact: upon considering profile functions $g_{\eps}$ of the type $g_{\eps}(\bar x)=\epsilon^{\alpha}b(\frac{\bar x}{\epsilon})$, where $b$ is a non-constant periodic function, we could apply  \cite[Lemma 6.2]{ArrLamb} to the polyharmonic problem in a straightofrward way; however, this would only guarantee the spectral stability  for $\alpha > m-1/2$. Our improved stability Theorem\ref{lemma: 6.2 improved} guarantees the spectral stability  for the better range  $\alpha >m-k+1/2$.

The proof of Theorem \ref{thm: poly strong}(ii) is based on a consequence of a degeneration argument that was introduced in \cite{CasDiazLuLaySuGrau}, and which was already exploited in \cite{ArrLamb}.

The reader may wonder if it is possible to push the arguments contained in the proof of Theorem \ref{thm: poly strong} in order to discuss the general case of weak intermediate problems for  polyharmonic operators. The main issue is that the degeneration argument in Theorem~\ref{thm: poly strong}$(ii)$ is restricted to the case of $(SIBC)$. Hence, a detailed analysis of the various possible situations seems to us  much more involved and almost prohibitive for arbitrary values of $m$  and $k$. We mention that the case $m=3$, $k=1$ will be the object of a forthcoming paper and we refer to \cite{ferraresso} for a number of results in this direction.

We remark that our main results, in particular Theorem \ref{lemma: 6.2 improved} and Theorem~\ref{thm: poly strong}, are  based on the notion of $\E$-convergence in the sense of Vainikko \cite{VainikkoSurvey} which is related to Stummel's discrete convergence  and to Anselone and Palmer's collective compactness, see   \cite{Stummel} and  \cite{AnsPalm} respectively. For a recent  survey on these topics and  further generalisations, we refer to  \cite{Boegli2}.

Finally, we mention that, in the case of second-order operators, counterexamples to the spectral stability with respect to domain perturbation are well-known, see for example the classical \cite[Chp. VI, 2.6]{C-H}. Related problems for the Neumann Laplace operator and for the Schr\"odinger operator with Neumann boundary conditions have been considered in \cite{AHH, CarKhrabu} and \cite{Arr1} respectively. Regarding higher order elliptic operators on variable domains, several contributions can be found in \cite{ArrFerLamb, BuosoLamb, BuosoLamb2, BuosoProv,BurLambJDE,BurLambSSIM, bulacompl, FeLa}. In particular, for a possible approach to these topics  via asymptotic analysis,  we refer to the articles \cite{BuCaNa, CosDRivDauMus, GomLobPer} and to the monographs \cite{MazNazPlaI, MazNazPlaII}. We refer also to the monograph \cite{GazzGS} and the articles \cite{FGazz, Prov} where polyharmonic operators are considered. For a wider discussion about perturbation theory for linear operators we mention the  monographs  \cite{Henry, Kato, necas}. 

This paper is organised as follows. Section 2 is devoted to preliminaries and notation, in particular to the definition of the class of operators and open sets under consideration. Section 3 contains a general discussion concerning the spectral stability of elliptic operators, and the proof of Theorem \ref{lemma: 6.2 improved}  and its corollaries, see in particular Theorem~\ref{thm: spectral stability polyharmonics}. In Section 4 we prove a Polyharmonic Green Formula which is used in the sequel and has its own interest. Section  5 is devoted to the analysis  of strong intermediate boundary conditions and to the proof of Theorem~\ref{thm: poly strong}. In the Appendix  we prove a technical lemma used in the proof of Theorem~\ref{thm: poly strong}(iii).


\section{Preliminaries and notation}
\label{sec: high variational}

In the sequel, we will use the following basic notation:
\begin{itemize}
\item $\N$ denotes the set of positive integers. Moreover, $\N_0 := \N \cup \{0\}$;
\item Given a normed space $X$, ${\mathcal{L}}(X)$ is the space of bounded linear operators on $X$;
\item If not otherwise specified, $m \in \N$ will always be greater or equal to 2;
\item $\Omega$, $\Omega_\eps$, $\eps_0 \geq \eps > 0$ will always denote bounded domains (i.e., open connected open sets in $\R^N$);
\item The standard Sobolev spaces with summability order 2 and smoothness order $m$ are denoted by $W^{m,2}_0(\Omega)$ and $W^{m,2}(\Omega)$.
\item The notation $V(\Omega), V(\Omega_\eps)$ will often be used for subspaces of $W^{m,2}(\Omega)$ (resp. $W^{m,2}(\Omega_\eps)$), containing $W_0^{m,2}(\Omega)$ (resp. $W^{m,2}_0(\Omega_\eps)$). 
\end{itemize}

\subsection{Classes of operators}
Let $M$ be the number of multiindices $\alpha = (\alpha_1, \dotsc, \alpha_N) \in \N_0^N$ with length $|\alpha| = \alpha_1 + \cdots + \alpha_N = m$. For all $\alpha, \beta \in \N_0^N$ such that $|\alpha| = |\beta| = m$, let $A_{\alpha \beta}$ be bounded measurable real-valued functions defined on $\R^N$ satisfying $A_{\alpha\beta} = A_{\beta \alpha}$ and the condition
\begin{equation}
\label{chp_hi: ell_cond}
\sum_{|\alpha| = |\beta| = m} A_{\alpha \beta}(x) \xi_\alpha \xi_\beta \geq 0,
\end{equation}
for all $x \in \R^N$, $(\xi_\alpha)_{|\alpha|= m } \in \R^M$. For all open subsets $\Omega$ of $\R^N$ we define
\begin{equation}
\label{eq: Q}
Q_\Omega(u,v) = \sum_{|\alpha|= |\beta| = m} \int_{\Omega} A_{\alpha \beta} D^\alpha u D^\beta v \, dx + \int_{\Omega} uv \, dx,
\end{equation}
for all $u,v \in W^{m,2}(\Omega)$ and we set $Q_{\Omega}(u) = Q_{\Omega}(u,u)$. Note that by \eqref{chp_hi: ell_cond} $Q_\Omega$ is a positive quadratic form, densely defined in the Hilbert space $L^2(\Omega)$. Hence, $Q_\Omega(\cdot, \cdot)$ defines a scalar product in $W^{m,2}(\Omega)$.

Let $V(\Omega)$ be a linear subspace of $W^{m,2}(\Omega)$ containing $W^{m,2}_0(\Omega)$. By standard Spectral Theory, if $V(\Omega)$ is complete with respect to the norm $Q^{1/2}_\Omega$, then there exists a uniquely determined non-negative self-adjoint operator $H_{V(\Omega)}$ such that $\D(H^{1/2}_{V(\Omega)}) = V(\Omega)$ and
\begin{equation}
\label{relation Q and H^1/2}
Q_\Omega(u,v) = \big(H^{1/2}_{V(\Omega)}u ,\, H^{1/2}_{V(\Omega)}v \big)_{L^2(\Omega)}, \quad\quad \textup{for all $u,v \in V(\Omega)$.}
\end{equation}
By \cite[Lemma 4.4.1]{Daviesbook} it follows that the domain $\D(H_{V(\Omega)})$ of $H_{V(\Omega)}$ is the subset of $W^{m,2}(\Omega)$ containing all the functions $u \in V(\Omega)$ for which there exists $f \in L^2(\Omega)$ such that
\begin{equation}
\label{chp_hi: weak Q}
Q_\Omega(u,v) = (f, v)_{L^2(\Omega)}, \quad \textup{for all $v \in V(\Omega)$,}
\end{equation}
in which case $H_{V(\Omega)} u = f$. If $u$ is a smooth function satisfying identity \eqref{chp_hi: weak Q} and the coefficients $A_{\alpha \beta}$ are smooth, by integration by parts it is immediate to verify that \eqref{chp_hi: weak Q} is the weak formulation of problem
$
Lu = f$ in $\Omega$,
where $L$ is the operator defined by
\begin{equation*}
Lu = (-1)^m \sum_{|\alpha| = |\beta| = m} D^\alpha(A_{\alpha\beta} D^\beta u) + u,
\end{equation*}
and the unknown $u$ is subject to suitable boundary conditions depending on the choice of $V(\Omega)$. 

If the embedding $V(\Omega) \subset L^2(\Omega)$ is compact, then the operator $H_{V(\Omega)}$ has compact resolvent. Consequently, its spectrum is discrete, and it consists of a sequence of isolated eigenvalues $\lambda_n[V(\Omega)]$ of finite multiplicity diverging to $+\infty$. By \cite[Theorem 4.5.3]{Daviesbook} the eigenvalues $\lambda_n[V(\Omega)]$ are determined by the following Min-Max principle:
\[
\lambda_n[V(\Omega)] = \min_{\substack{E \subset V(\Omega)\\ \text{dim}E = n}} \max_{\substack{u \in E \\ u \neq 0}} \frac{Q_\Omega(u)}{\norma{u}^2_{L^2(\Omega)}},
\]
for all $n \geq 1$. Furthermore, there exists an orthonormal basis in $L^2(\Omega)$ of eigenfunctions $\varphi_n[V(\Omega)]$ associated with the eigenvalues $\lambda_n[V(\Omega)]$.\\
We remark that in our assumptions there exist  two positive constants $c, C \in \R$ independent of $u$ such that
\[
c \norma{u}_{W^{m,2}(\Omega)} \leq Q^{1/2}_\Omega(u) \leq C \norma{u}_{W^{m,2}(\Omega)},
\]
which means that the two norms $Q^{1/2}_\Omega$ and $\norma{\cdot}_{W^{m,2}(\Omega)}$ are equivalent on $V(\Omega)$. Note that in general the constant $c$ may depend on $\Omega$. However, if the coefficients $A_{\alpha \beta}$ satisfy the uniform ellipticity condition
\begin{equation}
\label{def: unif_ellip}
\sum_{|\alpha| = |\beta| = m} A_{\alpha \beta}(x) \xi_\alpha \xi_\beta \geq \theta \sum_{|\alpha|=m} |\xi_{\alpha}|^2,
\end{equation}
for all $x \in \R^N$, $(\xi_\alpha)_{|\alpha|= m } \in \R^M$ and for some $\theta > 0$, then $c$ can be chosen independent of $\Omega$. \\

\subsection{Classes of open sets}

We recall the following definition from \cite[Definition 2.4]{BurLambSSIM} where for any given set $V\in\mathbb{R}^N$ and $\delta>0$, $V_\delta$ is the set $\{x\in \mathbb{R}^N:  \, d(x,\partial\Omega)>\delta\}$, and by a cuboid we mean any rotation of a rectangular parallelepiped in $\mathbb{R}^N$.

\begin{definition}\label{atlas}
Let $\rho>0$, $s,s'\in \mathbb{N}$ with $s'<s$. Let also $\{V_j\}_{j=1}^s$ be a family of bounded open cuboids and $\{r_j\}_{j=1}^s$ be a family of rotations in $\mathbb{R}^N$.  We say that $\mathcal{A}=(\rho, s,s',\{V_j\}_{j=1}^s,\{r_j\}_{j=1}^s)$ is an atlas in $\mathbb{R}^N$ with parameters $\rho, s,s',\{V_j\}_{j=1}^s,\{r_j\}_{j=1}^s$, briefly an atlas in $\mathbb{R}^N$.  Moreover, we consider the family of all open sets $\Omega\subset \mathbb{R}^N$ satisfying the following:
\par i) $\Omega\subset \cup_{j=1}^s (V_j)_\rho$ and $(V_j)_\rho\cap \Omega\ne \emptyset$
\par ii) $V_j\cap\partial\Omega\ne \emptyset$ for $j=1,\dots, s'$ and $V_j\cap\partial\Omega=\emptyset$ for $s'<j\leq s$
\par iii) for $j=1,\dots,s$ we have
\begin{align*}
&r_j(V_j)=\{ x\in \mathbb{R}^N:   a_{ij}<x_i<b_{ij}, i=1,\dots, N\}, \quad  j=1,\dots, s \\
&r_j(V_j\cap \Omega)=\{ x\in \mathbb{R}^N:   a_{Nj}<x_N<g_j(\bar x)\}, \quad\quad\:\:\:\: j=1,\dots, s'
\end{align*}
where $\bar x=(x_1,\dots, x_{N-1})$, $W_j=\{ x\in \mathbb{R}^{N-1}\! : a_{ij}<x_i<b_{ij}, i=1,\dots, N-1\}$
and $g_j\in C^{k,\gamma}(W_j)$ for $j=1,\dots, s'$, with $k\in\mathbb{N}_0$ and $0\leq \gamma \leq 1$.  Moreover, for $j=1,\ldots, s'$ we have $a_{Nj}+\rho\leq g_j(\bar x)\leq b_{Nj}-\rho$, for all $\bar x\in W_j$.

We say that an open set $\Omega$ is of class $C^{k,\gamma}_M(\mathcal{A})$ if all the functions $g_j$, $j=1,\dots, s'$ defined above are of class $C^{k,\gamma}(W_j)$ and  $\|g_j\|_{C^{k,\gamma}(W_j)}\leq M$.  We say that an open set $\Omega$ is of class $C^{k,\gamma}(\mathcal{A})$ if it is of class $C^{k,\gamma}_M(\mathcal{A})$ for some $M>0$. Also, we say that an open set $\Omega$ is of class $C^{k,\gamma}$ if it is of class $C^{k,\gamma}_M(\mathcal{A})$ for some atlas $\mathcal{A}$ and some $M>0$. Finally, we  denote by $C^{k}$ the class $C^{k,0}$ for $k\in \mathbb{N}\cup\{0\}$.

\end{definition}

It is important to note that if $\Omega $ is a $C^{0}$ bounded open set  then the Sobolev space $W^{m,2}(\Omega )$ (and consequently all the spaces $W^{m,2}(\Omega ) \cap W_0^{k,2}(\Omega)$, $1 \leq k \leq m$) is compactly embedded in $L^2(\Omega )$, see e.g., Burenkov~\cite{Bur_book}. Moreover, by using a common atlas as in Definition \ref{atlas}, it is possible to define a distance.

\begin{definition} {\bf (Atlas distance)}
\label{dev}
Let ${\mathcal{A}} =(\rho , s, s', \{V_j\}_{j=1}^s , \{r_j\}_{j=1}^s )$ be an atlas in ${\mathbb{R}}^N$.
For all $\Omega_1 , \Omega_2\in C^m({\mathcal{A}})$  and for all $h=0,\dots , m$ we set
\begin{equation*}
d^{(h)}_{{\mathcal{A}}}(\Omega_1,\Omega_2)=\max_{j=1, \dots ,s'}\sup_{0\le |\beta | \le h}\sup_{(\bar x , x_N)\in r_j(V_j)}\left| D^{\beta }g_{1j}(\bar x) -D^{\beta } g_{2j}(\bar x)  \right|,
\end{equation*}
where $g_{1j}$, $g_{2j}$ respectively, are the functions describing the boundaries of $\Omega_1, \Omega_2$ respectively, as in Definition \ref{atlas}.
Moreover,  we set $d_{{\mathcal{A}}}=d^{(0)}_{{\mathcal{A}}}$ and we call $d_{{\mathcal{A}}}$  `atlas distance'.
\end{definition}

\subsection{Formulae for higher order derivatives of composite functions}
We recall here few  well-known multidimensional formulae for the derivatives of composite functions. We will use the following notation: by  $\p(A)$ we denote the set of all subsets of a given finite non-empty set $A$ and by $\rm{Part}(A)$ we denote the set of all possible partitions of $A$. Namely, $\pi \in \rm{Part}(A)$ is a set the elements of which are pairwise disjoint subsets of $A$ whose union is $A$. Given $n \in \N$, we often write $\rm{Part(n)}$ in place of $\rm{Part}(\{1,\dots, n\})$ and $\p(n)$ in place of $\p(\{1,\dots,n\})$. Moreover we use the symbol $|A|$ to denote the cardinality of $A$; hence, for example $|\pi|$ with $\pi \in \rm{Part}(A)$ is the number of subsets of $A$ in the partition $\pi$.
Let $\Omega$ be an open set in $\R^N$.  If $I$ is an open set in $\R$ and  $f$ is  a $C^{n}$-function from $I$ to $\R$ and $\Phi$ is  a $C^{n}$ function from $\Omega$ to $I$,  then the  Fa\`{a} di Bruno formula reads
\begin{equation}
\label{eq: Faa di Bruno}
\frac{\partial^n f(\Phi(x))}{\partial x_{i_1} \cdots \partial x_{i_n}} = \sum_{\pi \in {\rm Part}(n)} f^{(|\pi|)}(\Phi(x)) \prod_{S \in \pi} \frac{\partial^{|S|} \Phi(x)}{\prod_{j \in S} \partial x_{i_j}}.
\end{equation}
Moreover, the  Leibnitz formula for the derivatives of the product of two functions $u,v$ of class $C^n(\Omega)$ can can be written as follows
\begin{equation}
\label{eq: Leibnitz}
\frac{\partial^n (uv)}{\partial x_{i_1} \cdots \partial x_{i_n}} = \sum_{S \in \p(n)} \frac{\partial^{|S|} u}{\prod_{j \in S} \partial x_{i_j}} \frac{\partial^{(n-|S|)}v}{\prod_{j \notin S} \partial x_{i_j}},
\end{equation}
where  $j \notin S$ means that $j$ lies in the complementary of $S$ in $\{1, \dotsc, n\}$. We recall that in general,  if   $\Phi$ is  a $C^{n}$ function from an open subset   $U$ of $\R^N$     to an open subset $V$ of $\R^r$, and $f$ is a function in  $W^{n,1}_{loc}(V)$   then the Fa\`{a} di Bruno formula reads
\begin{small}
\begin{equation}
\label{eq: Faa di Bruno mdim}
\frac{\partial^n f(\Phi(x))}{\partial x_{i_1} \cdots \partial x_{i_n}} = \sum_{\pi \in {\rm Part}(n)}\sum_{j_1, \dots , j_{|\pi |}\in \{1, \dots , r \} }
\frac{\partial^{|\pi|} f}{\prod_{k=1}^{|\pi|} \partial x_{j_k}}(\Phi (x))\, \prod^{|\pi|}_{k=1} \frac{\partial^{|S_k|} \Phi^{(j_k)}}{\prod_{l \in S_k} \partial x_{i_l}} 
\end{equation}
\end{small}

\section[Spectral convergence on domains with perturbed boundaries]{Higher order operators on domains with perturbed boundaries}
\label{sec: setting homogeniz}

Let $m \in \N$, $m \geq 2$ and let $\eps > 0$. Let $V(\Omega), V(\Omega_\eps)$ be subspaces of $W^{m,2}(\Omega)$, $W^{m,2}(\Omega_\eps)$ respectively, containing $W_0^{m,2}(\Omega)$, $W^{m,2}_0(\Omega_\eps)$  respectively. Moreover, let $H_{V(\Omega)}$, $H_{V(\Omega_\eps)}$, $Q_\Omega$, $Q_{\Omega_\eps}$ be as in \eqref{relation Q and H^1/2}. A fundamental part of our analysis will be based on the following:

\begin{definition}\label{def: cond C}
(\cite[Definition 3.1]{ArrLamb}). Given open sets $\Omega_\eps$, $\eps > 0$ and $\Omega \in \R^N$ with corresponding elliptic operators $H_{V(\Omega_\eps)}$, $H_{V(\Omega)}$ defined on $\Omega_\eps$, $\Omega$ respectively, we say that condition $(C)$ is satisfied if there exists open sets $K_\eps \subset \Omega \cap \Omega_\eps$ such that
\begin{equation}
\label{def: condition C null measure}
\lim_{\eps \to 0} |\Omega \setminus K_\eps | = 0,
\end{equation}
and the following conditions are satisfied:

\noindent (C1) If $v_\eps \in V(\Omega_\eps)$ and $\sup_{\eps >0} Q_{\Omega_\eps}(v_\eps) < \infty$ then $\lim_{\eps \to 0} \norma{v_\eps}_{L^2(\Omega_\eps \setminus K_\eps)} = 0$.\\
\noindent (C2) For each $\eps >0$ there exists an operator $T_\eps$ from $V(\Omega)$ to $V(\Omega_\eps)$ such that for all fixed $\varphi \in V(\Omega)$
\begin{enumerate}[label=(\roman*)]
\item $\lim_{\eps \to 0} Q_{K_\eps}(T_\eps \varphi -\varphi) = 0$;
\item $\lim_{\eps \to 0} Q_{\Omega_\eps \setminus K_\eps}(T_\eps \varphi) = 0$;
\item $\lim_{\eps \to 0} \norma{T_\eps \varphi}_{L^2(\Omega_\eps)} < \infty$.
\end{enumerate}

\noindent (C3) For each $\eps > 0$ there exists an operator $E_\eps$ from $V(\Omega_\eps)$ to $W^{m,2}(\Omega)$ such that the set $E_\eps(V(\Omega_\eps))$ is compactly embedded in $L^2(\Omega)$ and such that
\begin{enumerate}[label=(\roman*)]
\item If $v_\eps \in V(\Omega_\eps)$  is a sequence such that $\sup_{\eps > 0} Q_{V(\Omega_\eps)}(v_\eps) < \infty$, then $\lim_{\eps \to 0} Q_{K_\eps}(E_\eps v_\eps -v_\eps) = 0$;
\item $$\sup_{\eps >0} \sup_{v \in V(\Omega_\eps)\setminus \{0\}} \frac{\norma{E_\eps v}_{W^{m,2}(\Omega)}}{Q^{1/2}_{\Omega_\eps}(v)} < \infty;$$
\item If $v_\eps \in V(\Omega_\eps)$ for all $\eps > 0$, $\sup_{\eps > 0} Q_{\Omega_\eps}(v_\eps) < \infty$ and there exists $v \in L^2(\Omega)$ such that, up to a subsequence, we have $\norma{E_\eps v_\eps - v}_{L^2(\Omega)} \to 0$, then $v \in V(\Omega)$.
\end{enumerate}
\end{definition}
It is proved  in  \cite[Theorem 3.5]{ArrLamb} that Condition (C) guarantees the spectral convergence of the operators $H_{V(\Omega_\eps)}$ to the operator  $H_{V(\Omega)}$ as $\eps\to 0$.\\
The convergence of the operators is understood in the sense of the compact convergence, as defined in \cite{VainikkoSurvey}. Let us briefly recall the setting. Let ${\mathcal E}$ be the \emph{extension-by-zero operator}, mapping any given  real-valued function $u$ defined on some subset $A$ of ${\mathbb{R}}^N$, to the function ${\mathcal E}u$  such that ${\mathcal E}u = u$ a.e. in $A$ and ${\mathcal E}u = 0$ a.e. in $\R^N \setminus A$.  By using ${\mathcal E}$  we can map functions in $L^2(\Omega )$ to the space $L^2(\Omega _{\epsilon })$, for every $\eps > 0$, so that $\E$ defines a \emph{``connecting system''} between $L^2(\Omega)$ and the family of spaces $(L^2(\Omega_\eps))_{\eps > 0}$. We then say that:
\begin{itemize}
\item $v_\eps\in L^2(\Omega_\eps)$ \emph{$\mathcal{E}$-converges} to $v\in L^2(\Omega)$ if $\|v_\eps-\mathcal{E}v\|_{L^2(\Omega_\eps)}$ $\to 0$ as $\eps\to 0$;
\item a family of bounded linear operators $B_\eps\in \mathcal{L}(L^2(\Omega_\eps))$ \emph{$\mathcal{EE}$- converges} to $B\in \mathcal{L}(L^2(\Omega))$ if $B_\eps v_\eps$ $\mathcal{E}$-converges to $Bv$ whenever $v_\eps$ $\mathcal{E}$-converges to $v$;
\item a family of bounded, compact linear operators $B_\eps\in \mathcal{L}(L^2(\Omega_\eps))$ is said to \emph{$\mathcal{E}$-compact converges} to $B\in \mathcal{L}(L^2(\Omega))$ if $B_\eps$ $\mathcal{EE}$-converges to $ B$ and for any family of functions $v_\eps\in L^2(\Omega_\eps)$ with $\|v_\eps\|_{L^2(\Omega_\eps)}\leq 1$ there exists a subsequence, denoted by $v_\eps$ again, and a function $w\in L^2(\Omega)$ such that $B_\eps v_\eps$ $\mathcal{E}$-converges to $ w$.
\end{itemize}
We refer to \cite[Section 2.2]{ArrLamb}, for further information on this type of convergence. Importantly, in our assumptions on the operators $H_{V(\Omega_\eps)}$, $H_{V(\Omega)}$, the compact convergence of the resolvent operators is a sufficient condition for the spectral convergence. In particular, we have the following

\begin{theorem}\label{arlathm}
Let $\Omega_\eps$, $\eps>0$ and $\Omega$  be open sets in $\R^N$. Let $H_{V(\Omega_\eps)}$, $H_{V(\Omega)}$ be operators with compact resolvents, associated with $V(\Omega_\eps)$, $V(\Omega)$, respectively, as in \eqref{relation Q and H^1/2}, such that condition $(C)$ is satisfied. Let $\lambda_k$, $\lambda_k^\eps$ be the $k$-th eigenvalue of $H_{V(\Omega)}$, $H_{V(\Omega_\eps)}$, respectively. Then $H^{-1}_{V(\Omega_\eps)}$ $\mathcal{E}$-compact converges to $ H^{-1}_{V(\Omega)}$ as $\eps \to 0$. Moreover,
\begin{itemize}
\item[(i)] $\lambda_n^\eps\to \lambda_n$ as $\eps\to 0$, for all $n\in \N$.
\item[(ii)] If $\lambda _n=\lambda _{n+1}=\dots =\lambda _{n+h-1}$ is an eigenvalue of multiplicity $h$ and $\varphi^\eps_n$, $\varphi^\eps_{n+1},\dots , \varphi_{n+h-1}^\eps$ is an orthonormal set in $L^2(\Omega_{\eps})$ of eigenfunctions  associated with the corresponding eigenvalues $\lambda^\eps_n, $ $ \lambda _{n+1}^\eps,$ $\dots ,$ $ \lambda _{n+h-1}^\eps$, then there exists  an orthonormal set  $\varphi_n$, $\varphi_{n+1}$, $\dots , \varphi_{n+h-1}$ in $L^2(\Omega )$ of eigenfunctions  associated with the eigenvalues $(\lambda_{n+t-1})_{t=1}^h$ such that, possibly passing to a suitable subsequence,   $\varphi_{n+i-1}^\eps$   $\mathcal{E}$-converges to $\varphi_{n+i-1}$ as $\eps \to 0  $ for all $i=1,\dots , h$.
  \item[(iii)] If $\lambda _n=\lambda _{n+1}=\dots =\lambda _{n+h-1}$ is an eigenvalue of multiplicity $h$  and  $\varphi_n$, $\varphi_{n+1}$, $\dots , \varphi_{n+h-1}$ is an orthonormal set $L^2(\Omega )$ of eigenfunctions  associated with $(\lambda_{n+t-1})_{t=1}^h$ then for every $\eps >0$ there exists an orthonormal set in $L^2(\Omega_{\eps})$ of eigenfunctions  $\varphi^\eps_n$, $\varphi^\eps_{n+1},\dots , \varphi^\eps_{n+h-1}$  associated with the corresponding eigenvalues $\lambda^\eps_n, \lambda _{n+1}^\eps,$ $\dots ,\lambda _{n+h-1}^\eps$ such that $\varphi_{n+i-1}^\eps$   $\mathcal{E}$-con\-ver\-ges to
  $\varphi_{n+i-1}$ as $\eps \to 0  $ for all $i=1,\dots , h$.
\end{itemize}
\end{theorem}

When the claims $(i)-(ii)-(iii)$ of the previous theorem are verified, we say that $H_{V(\Omega_\eps)}$ \emph{spectrally converges} to $H_{V(\Omega)}$ as $\eps \to 0$.

\subsection{An explicit condition for the spectral stability}
We consider now the following geometric setting:\\[0.2cm]
\textbf{(G1)} There exists a cuboid $V$ of the form $W \times (a,b)$, where $W \subset \R^{N-1}$ is an open, connected and bounded set of class $C^m$, and $g,g_\epsilon \in C^m(\overline{W})$ such that
\begin{align}
\Omega \cap V &= \{(\bar{x}, x_N) \in W \times (a,b) : a < x_N < g(\bar{x}) \}, \label{scatola1} \\
\Omega_\epsilon \cap V &= \{(\bar{x}, x_N) \in W \times (a,b) : a < x_N < g_\epsilon(\bar{x}) \}.\label{scatola2}
\end{align}
Assume that $\Omega \setminus V = \Omega_\epsilon \setminus V$ for all $\eps > 0$.

\vspace{0.2cm}

It is convenient to set $\Omega_0=\Omega$. According to Def.~\ref{atlas}, if $\Omega_\eps \in C^m({\mathcal{A}})$ for all $\eps \geq 0$, then we can assume \textbf{(G1)} without loss of generality. For all $\eps\geq 0$, let us consider the quadratic forms $Q_{\Omega_{\eps}}$  on $\Omega_{\eps}$ defined as in \eqref{eq: Q}, where the coefficients $A_{\alpha\beta}$ are independent of $\eps\geq 0$ and satisfy the uniform ellipticity condition \eqref{def: unif_ellip}. Then we consider the non-negative self-adjoint operators $H_{V(\Omega_{\eps})}$ defined by \eqref{relation Q and H^1/2} with $V(\Omega)$ replaced by $V(\Omega_{\eps}) = W^{m,2}(\Omega_{\eps}) \cap W_0^{k,2}(\Omega_{\eps})$ for some $1\leq k < m$. Since $\Omega_{\eps}$ is of class $C^m$, $V(\Omega_{\eps})$ is compactly embedded in $L^2(\Omega_{\eps})$ hence  $H_{V(\Omega_{\eps})}$ has compact resolvent.

We now state our first result, concerning an explicit condition sufficient to guarantee the spectral convergence of the operators $H_{V(\Omega_{\eps})}$. 
This theorem is a generalisation of \cite[Lemma 6.2]{ArrLamb}.

\begin{theorem}
\label{lemma: 6.2 improved}
Let $\Omega_\eps$, $\eps \geq 0$ satisfy assumption \textbf{(G1)}. Suppose that for some $k\in \N$, with   $1\leq k < m$,  $V(\Omega_{\eps}) = W^{m,2}(\Omega_{\eps}) \cap W^{k,2}_0(\Omega_{\eps})$ for all $\eps\geq 0$.
If for all $\epsilon>0$ there exists $\kappa_\epsilon >0$ such that
\begin{enumerate}[label=(\roman*)]
\item $\kappa_\epsilon > \norma{g_\epsilon - g}_\infty$, \quad $\forall \epsilon > 0$, \quad $\lim_{\epsilon \to 0} \kappa_\epsilon  = 0$, 
\item $\lim_{\epsilon \to 0} \frac{\norma{D^\beta(g_\epsilon - g)}_\infty}{\kappa_\epsilon^{m-\abs{\beta}-k+1/2}} = 0$, \:$\forall \beta \in \N_0^N$ with $|\beta |\leq m$,
\end{enumerate}
then $H^{-1}_{V(\Omega_\epsilon)} $ $\mathcal{E}$-compact converges to $ H^{-1}_{V(\Omega)}$ as $\eps \to 0$. In particular, $H_{V(\Omega_\epsilon)}$ spectrally converges to $H_{V(\Omega)}$ as $\eps \to 0$
\end{theorem}
\begin{proof}
We first observe that the last statement is a direct consequence of Theorem \ref{arlathm}. The case $k=1$ is proved in \cite[Lemma 6.2]{ArrLamb}. Thus, we suppose $k>1$. It is possible to assume directly that $\Omega=\Omega\cap V$ and $\Omega_{\eps}=\Omega_{\eps}\cap V$ as in \eqref{scatola1} and \eqref{scatola2} respectively.
Define $k_\eps = M \kappa_\eps$ for a suitable constant $M> 2m$. Let $\tilde{g}_\eps = g_\eps - k_\eps$ and
\[
K_\eps = \{(\bar{x},x_N) \in W \times ]a,b[ : a < x_N < \tilde{g}_\eps(\bar{x})  \}.
\]
Note that with this definition of $K_\eps$ \eqref{def: condition C null measure} is satisfied. By the standard one dimensional estimate
\begin{equation}
\label{proof: boundedness Sob}
\norma{f}_{L^\infty(a,b)} \leq C \norma{f}_{W^{1,2}(a,b)},
\end{equation}
and Tonelli Theorem it follows that condition (C1) is satisfied.\\
We now  define a suitable family of diffeomorphisms  $\Phi_\eps : \overline{\Omega}_\eps \to \overline{\Omega}$ by setting
$$\Phi_\eps(\bar{x}, x_N) = (\bar{x}, x_N - h_\eps(\bar{x}, x_N)),$$
for all $(\bar{x}, x_N) \in \overline{\Omega}_\eps$, where
\[
h_\eps(\bar{x}, x_N) =
\begin{cases}
0, &\textup{if $a \leq x_N \leq \tilde{g}_\eps(\bar{x})$,}\\
(g_\eps(\bar{x}) - g(\bar{x})) \bigg( \frac{x_N - \tilde{g}_\eps(\bar{x})}{g_\eps(\bar{x}) - \tilde{g}_\eps(\bar{x})} \bigg)^{m+1} &\textup{if $\tilde{g}_\eps(\bar{x}) < x_N \leq g_\eps(\bar{x})$.}
\end{cases}
\]
Then consider the map  $T_\eps$  from $V(\Omega)$ to $V(\Omega_\eps)$ defined by
\begin{equation}
\label{def: T eps high}
T_\eps \varphi = \varphi \circ \Phi_\eps,
\end{equation}
for all $\varphi \in V(\Omega)$. One can check that $T_\eps$ is well-defined and that condition (C2)(i) is satisfied. We now want to prove that conditions (C2)(ii), (iii) are satisfied. We need to estimate the derivatives of $\varphi \circ \Phi_\epsilon$. Here we can improve the estimate given in \cite[Lemma 6.2]{ArrLamb} by taking advantage of the decay of $D^\gamma\varphi$ in a neighbourhood of $\partial \Omega$, for $\abs{\gamma} \leq k-1$. We divide the proof in two steps.\smallskip

\textbf{Step 1.} We aim at proving  a decay estimates for the $L^2$-norms of the derivatives of $\varphi$ near the boundary, namely estimate \eqref{eq: ineq decay general}. First, note that
 \[
 \Phi_\epsilon(\Omega_\epsilon \setminus K_\epsilon) = \Omega \setminus K_\eps = \{ (\bar{x},x_N) \in \Omega :\, \bar{x}\in W,\: g_\epsilon(\bar{x}) - k_\epsilon \leq x_N \leq g(\bar{x}) \},
 \]
 for any $\epsilon > 0$. Fix $x \in \Phi_\epsilon(\Omega_\epsilon \setminus K_\epsilon)$ and $\beta \in \N_0^N$, $\abs{\beta} \leq k-1$. Suppose for the moment $\varphi \in C^m(\overline{\Omega})$. By the Taylor's formula with remainder in integral form, we get that
\[  \label{taylor}
D^{\beta}
\varphi(x)
= \sum_{l=0}^{k-1-\abs{\beta}} \frac{1}{l!} \frac{\partial^{l}{(D^\beta \varphi(\bar{x},g(\bar{x})))}}{\partial{x_N^l}} (x_N - g(\bar{x}))^l + R(\beta,x),
\]
where
\begin{small}
\[
R(\beta,x):= \frac{(x_N - g(\bar{x}))^{k-\abs{\beta}}}{(k-\abs{\beta}-1)!} \int_0^1(1-t)^{k-1-\abs{\beta}} \frac{\partial^{k-\abs{\beta}}}{\partial{x_N^{k-\abs{\beta}}}} D^\beta \varphi(\bar{x}, g(\bar{x}) + t (x_N - g(\bar{x})) \, \diff{t}\, .
\]
\end{small}
Note that $ -2k_\epsilon \leq g_\epsilon(\bar{x}) - g(\bar{x}) -k_\epsilon \leq x_N - g(\bar{x}) \leq 0$. By Jensen's inequality,
\begin{equation}
\label{proof: estimate remainder}
|R(\beta,x)|^2 \leq (2 k_\epsilon )^{2(k-\abs{\beta})} \int_0^1 \Bigg\lvert\frac{\partial^{k-\abs{\beta}}}{\partial{x_N^{k-\abs{\beta}}}} D^\beta \varphi(\bar{x}, g(\bar{x}) + t (x_N - g(\bar{x}))\Bigg\rvert^2 \, \diff{t}.
\end{equation}
An integration in the variable $x_N$ in \eqref{proof: estimate remainder} and inequality \eqref{proof: boundedness Sob} applied to the interval $(a,g(\bar x))$ yield
\begin{small}
\begin{equation}
\label{proof: estim remainder 2}
\begin{split}
\int_{g_\epsilon(\bar x) - k_\epsilon}^{g(\bar{x})}|R(\beta,x)|^2 \,\diff{x_N} \leq C k_\epsilon^{2(k-\abs{\beta}) + 1} \Bigg\lVert\frac{\partial^{k-\abs{\beta}+1}}{\partial{x_N^{k-\abs{\beta}+1}}} D^\beta \varphi(\bar{x}, \cdot)\Bigg\rVert_{W^{2,2}(a,g(\bar{x}))}^2
\end{split}
\end{equation}
\end{small}
By integrating both sides of \eqref{proof: estim remainder 2} with respect to $\bar{x}\in W$, we finally get
\begin{equation}
\label{eq: ineq decay}
 \int_{\Phi_\epsilon(\Omega_\epsilon \setminus K_\epsilon)} |R(\beta,x)|^2 \,\diff{x} \leq C k_\epsilon^{2(k-\abs{\beta}) + 1} \norma{\varphi}^2_{W^{m,2}(\Omega)},
\end{equation}
for sufficiently small $\epsilon$, for all $|\beta| \leq k-1$. Thus,  by \eqref{taylor} we get
\begin{equation}
\label{proof: estim derivatives}
\begin{split}
&\int_{\Phi_\epsilon(\Omega_\epsilon \setminus K_\epsilon)} |D^{\beta}\varphi(x)|^2 \,\diff{x}\leq C k_\epsilon^{2(k-\abs{\beta}) + 1} \norma{\varphi}^2_{W^{m,2}(\Omega)}\\
&+ C \int_W \int_{g_\epsilon (\bar x)- k_\epsilon}^{g(\bar{x})}\Bigg\lvert\sum_{l=0}^{k-1-\abs{\beta}} \frac{\partial^{l}{(D^\beta \varphi(\bar{x},g(\bar{x}))}}{\partial{x_N^l}}\Bigg\rvert^2 |x_N - g(\bar{x})|^{2l}\diff{\bar{x}}\,\diff{x_N}\, ,
\end{split}
\end{equation}
for all sufficiently small $\epsilon$, and $|\beta| \leq k-1$. We now estimate the last integral in the right-hand side of \eqref{proof: estim derivatives} in the following way
\begin{equation}
\label{proof: estim derivatives 2}
\begin{split}
&\sum_{l=0}^{k-1-\abs{\beta}}  \int_W \int_{g_\epsilon (\bar x) - k_\epsilon}^{g(\bar{x})}\Bigg\lvert\frac{\partial^{l}{(D^\beta \varphi(\bar{x},g(\bar{x}))}}{\partial{x_N^l}}\Bigg\rvert^2 |x_N - g(\bar{x})|^{2l}\diff{\bar{x}}\,\diff{x_N}\\
&\leq \sum_{l=0}^{k-1-\abs{\beta}} k_\epsilon^{2l+1} \int_W \Bigg\lvert\frac{\partial^{l}{(D^\beta \varphi(\bar{x},g(\bar{x}))}}{\partial{x_N^l}}\Bigg\rvert^2 \diff{\bar{x}}\\
&= \sum_{l=0}^{k-1-\abs{\beta}} C  k_\epsilon^{2l+1}   \Bigg\rVert \frac{\partial^{l}{(D^\beta \varphi)}}{\partial{x_N^l}} \Bigg\lVert^2_{L^2(\Gamma )},
\end{split}
\end{equation}
where $\Gamma:=\{(\bar x, g(\bar x)):\ \bar x\in W  \}$.
Thus, by \eqref{proof: estim derivatives}, \eqref{proof: estim derivatives 2} we obtain
\begin{equation}
\label{eq: ineq 1}
\begin{split}
&\int_{\Phi_\epsilon(\Omega_\epsilon \setminus K_\epsilon)} |D^{\beta}\varphi(x)|^2 \,\diff{x}\\
&\leq \sum_{l=0}^{k-1-\abs{\beta}} C k_\epsilon^{2l+1} \Bigg\rVert \frac{\partial^{l}{(D^\beta \varphi)}}{\partial{x_N^l}} \Bigg\lVert^2_{L^2(\Gamma )} + C k_\epsilon^{2(k-\abs{\beta}) + 1} \norma{\varphi}^2_{W^{m,2}(\Omega)}.
\end{split}
\end{equation}
Inequality~\eqref{eq: ineq 1} holds for smooth functions. If $\varphi \in W^{m,2}(\Omega) \cap W^{k,2}_0(\Omega)$, then we can choose a sequence $(\psi_n)_{n\geq 1} \subset C^{\infty}(\overline{\Omega})$ such that $\psi_n \to \varphi$ in $W^{m,2}(\Omega)$ (this is possible because $\partial \Omega$ is Lipschitz continuous). We then use~\eqref{eq: ineq 1} for $\psi_n$, and we pass to the limit as $n \to \infty$ by using the continuity of the trace operator and standard estimates on the intermediate derivatives of Sobolev functions (see e.g., \cite[\S4.4]{Bur_book}). We deduce that
\begin{equation}
\label{eq: ineq decay general}
\int_{\Phi_\epsilon(\Omega_\epsilon \setminus K_\epsilon)} |D^{\beta}\varphi(x)|^2 \,\diff{x} \leq C k_\epsilon^{2(k-\abs{\beta}) + 1} \norma{\varphi}^2_{W^{m,2}(\Omega)},
\end{equation}
for all sufficiently small $\epsilon$. Actually, inequality~\eqref{eq: ineq decay general} holds also for $|\beta| = k$ (possibly modifying the constant in the right hand side). Indeed, $D^\beta \varphi \in W^{2,2}(\Omega)$, for any $|\beta| = k$, hence
by standard boundedness of Sobolev functions on almost all vertical lines (see \eqref{proof: boundedness Sob}) we find that 
\[
\int_W \int_{g_\epsilon - k_\epsilon}^{g(\bar{x})}|D^{\beta}\varphi(x)|^2 \,\diff{x_N}\, \diff{\bar{x}}
\leq 2k_\epsilon \int_W \norma{D^\beta \varphi(\bar{x},\cdot) }^2_\infty \diff{\bar{x}}
\leq 2 C k_\epsilon \norma{\varphi }^2_{W^{m,2}(\Omega)}.
\]
This concludes Step 1.\smallskip

\textbf{Step 2.} We claim that Condition (C2)(ii) holds. Let $\varphi \in V(\Omega)$ and let $\alpha$ be a fixed multiindex such that $|\alpha|=m$. We write
\begin{equation}\label{eq: mortal polynomial}
D^\alpha \varphi (\Phi_\epsilon(x)) = \sum_{1 \leq |\beta| \leq m} D^\beta \varphi(\Phi_\epsilon(x)) p^\alpha_{m,\beta}(\Phi_\epsilon)(x),
\end{equation}
where $p^\alpha_{m,\beta}(\Phi_\epsilon) $ is a homogeneous polynomial of degree $|\beta|$ in derivatives of $\Phi_\epsilon$ of order not exceeding $m-|\beta|+1$. Note that the polynomial $p^\alpha_{m,\beta}(\Phi_\epsilon)$ appearing in~\eqref{eq: mortal polynomial} is the sum of several terms $\Theta$ in the following form
\[
\Theta = D^{k_1}\Biggl( \delta_{j_1,N} - \frac{\partial h_\eps}{\partial x_{j_1}}\Biggr) \cdots D^{k_n}\Biggl( \delta_{j_{n},N} - \frac{\partial h_\eps}{\partial x_{j_{n}}}\Biggr) \frac{\partial \Phi^{(i_{n+1})}}{\partial x_{i_{n+1}}} \cdots \frac{\partial \Phi^{( i_{|\beta|} )}}{\partial x_{i_{|\beta|}}},
\]
where\footnote{Here it is understood that for $|\beta|=1$ the terms $\frac{\partial \Phi^{(j_{n+1})}}{\partial x_{i_{n+1}}} \cdots \frac{\partial \Phi^{(j_{|\beta|})}}{\partial x_{i_{|\beta|}}}$ are not present; recall that $m \geq 2$.} $1 \leq n \leq |\beta|$, $1\leq j_i \leq N$ for all $i=1,\dots, n$, $i_{n+1}, \dots , i_{|\beta|} $ are in $\{1, \dots, N-1\}$, and $k_1,\dots,k_{n}$ are multiindexes satisfying $\abs{k_1} + \dots + \abs{k_{n}}= m-\abs{\beta}$. Moreover,  $\Theta$ is a sum of terms of the type $D^{L_1}h_\eps  \cdots  D^{L_l} h_\eps$, for all $1 \leq l \leq n$, for suitable  multiindexes $L_1,...,L_l$ satisfying
\begin{equation}\label{eq: sum of Li}
\abs{L_1} + \dots + \abs{L_l} = m- \abs{\beta} + l.
\end{equation}
Now by \cite[Inequality (6.7)]{ArrLamb} and hypothesis $(iii)$ we have
\[
\begin{split}
&\norma{D^{L_1}h_\eps \cdots D^{L_l} h_\eps}_\infty \\
&\leq C \Biggl(\sum_{\abs{\gamma_1} \leq \abs{L_1}} \frac{\norma{D^{\gamma_1}(g_\eps-g)}_\infty}{\kappa_\epsilon^{\abs{L_1} - \abs{\gamma_1}}} \Biggr) \cdots \Biggl(\sum_{\abs{\gamma_l} \leq \abs{L_l}} \frac{\norma{D^{\gamma_l}(g_\eps-g)}_\infty}{\kappa_\epsilon^{\abs{L_l} - \abs{\gamma_l}}} \Biggr)\\
&\leq  o(1) \Biggl(\sum_{\abs{\gamma_1} \leq \abs{L_1}} \frac{\kappa_\epsilon^{m-\abs{\gamma_1}-k+1/2}}{\kappa_\epsilon^{\abs{L_1} - \abs{\gamma_1}}} \Biggr) \cdots \Biggl(\sum_{\abs{\gamma_l} \leq \abs{L_l}} \frac{\kappa_\epsilon^{m-\abs{\gamma_l}-k+1/2}}{\kappa_\epsilon^{\abs{L_l} - \abs{\gamma_l}}} \Biggr)\\
&\leq  o(1) \kappa_\epsilon^{l(m-k+1/2) - \sum_i \abs{L_i}} =  o(1) \kappa_\epsilon^{l(m-k+1/2) - \sum_i \abs{L_i} - \abs{\beta} + k + 1/2} \cdot \kappa_\epsilon^{\abs{\beta} - k - 1/2}\\
&\leq  o(1) \kappa_\epsilon^{\abs{\beta} - k - 1/2}
\end{split}
\]
where the last inequality holds provided that
$$
l(m-k+1/2) - \sum_i \abs{L_i} - \abs{\beta} + k + 1/2 \geq 0.
$$
By~\eqref{eq: sum of Li}, we have to check that
$
l(m-k+1/2) - (m- \abs{\beta} + l) - \abs{\beta} + k + 1/2 \geq 0,
$
which is verified if and only if
$
l(m-k-1/2) \geq m-k-1/2,
$
and this holds true because $m-k-1/2 > 0$ and $l\geq 1$. Hence we have proved that
\begin{equation} \label{eq: estimate p}
\norma{p^\alpha_{m,\beta}(\Phi_\epsilon)}_\infty \leq \, o(1)\, \kappa_\epsilon^{\abs{\beta} - k - 1/2}.
\end{equation}

By inequalities~\eqref{eq: ineq decay general} and \eqref{eq: estimate p}, we deduce that
\begin{equation}
\label{proof: big est Q Omega eps}
\begin{split}
&Q_{\Omega_\epsilon\setminus K_\epsilon}( T_\epsilon\varphi) \leq \int_{\Omega_\epsilon\setminus K_\epsilon} \abs{\varphi(\Phi_\epsilon) }^2 \,\diff{x} + C \sum_{|\alpha| = m} \int_{\Omega_\epsilon\setminus K_\epsilon} \abs{D^\alpha \varphi(\Phi_\epsilon) }^2 \,\diff{x}\\
&\leq C \int_{\Phi_\eps(\Omega_\eps \setminus K_\epsilon)} \abs{\varphi}^2 \,\diff{x} + C \sum_{\substack{|\alpha| = m \\ 1\leq |\beta| \leq k}} \norma{p^\alpha_{m,\beta}(\Phi_\epsilon)}_\infty^2 \int_{\Omega_\epsilon\setminus K_\epsilon} \abs{D^\beta \varphi(\Phi_\epsilon(x)) }^2 \,\diff{x}\\
 &+ C \sum_{\substack{|\alpha| = m \\ k< |\beta| \leq m}} \norma{p^\alpha_{m,\beta}(\Phi_\epsilon)}_\infty^2 \int_{\Omega_\eps\setminus K_\epsilon} \abs{D^\beta \varphi(\Phi_\epsilon(x)) }^2 \,\diff{x}\\
&\leq C\norma{\varphi}^2_{L^2(\Omega \setminus K_\epsilon)} + o(1) \kappa_\epsilon^{2(|\beta|-k-1/2)} \kappa_\epsilon^{2(k-\abs{\beta}) + 1} + o(1)\norma{\varphi}^2_{W^{m,2}(\Omega\setminus K_\epsilon)},
\end{split}
\end{equation}
for all $\eps > 0$ sufficiently small. Since the right-hand side of \eqref{proof: big est Q Omega eps}  vanishes as $\epsilon \to 0$ we conclude that condition $(C2) (ii)$ is satisfied. \\
It remains to prove condition (C3). To prove that conditions (C3)(i), (C3)(ii) are satisfied it is sufficient to set $E_\eps u = ({\rm Ext}_{\Omega_\eps} u)|_{\Omega}$ for all $u \in V(\Omega_\eps)$, where ${\rm Ext}_{\Omega_\eps}$ is the standard Sobolev extension operator mapping $W^{m,2}(\Omega_\eps)$ to $W^{m,2}(\R^N)$. Finally, in order to prove condition (C3)(iii) it is sufficient to prove that the weak limit $v$ of the uniformly bounded sequence $v_\eps$ (appearing in the statement of condition (C3)(iii)) lies in $W^{k,2}_0(\Omega)$. This is easily achieved by considering the extension-by-zero of the functions $v_\eps$ outside $\Omega_\eps$, passing to the limit and recalling that the limit set $\Omega$ has Lipschitz boundary. \qedhere

\end{proof}
%

\vspace{0.3cm}

Theorem \ref{lemma: 6.2 improved} can be actually applied to open sets $\Omega $ in the atlas class  $C^m({\mathcal{A}})$  by requiring that the assumptions  of  Lemma \ref{lemma: 6.2 improved} are satisfied by all the profile functions $g_j$ describing their boundaries.  
%
%
%
Then we can prove the following

\begin{theorem}Let ${\mathcal{A}}$ be an atlas in ${\mathbb{R}}^N$, $M>0$, $m\in {\mathbb{N}}$, $m\geq 2$. For all $\eps \geq 0$, let $\Omega_{\eps} \in C^m_M(\mathcal{A})$. Let $k\in \N$   with   $1\leq k < m$ and define, for all $\eps\geq 0$,  $V(\Omega_{\eps}) = W^{m,2}(\Omega_{\eps}) \cap W^{k,2}_0(\Omega_{\eps})$.  If
$$
\lim_{\eps\to 0}d_{\mathcal{A}}^{(m-k)}(\Omega_{\eps}, \Omega )=0,
$$
then condition {\bf (C)} is satisfied, hence  $H^{-1}_{V(\Omega_{\epsilon})}$ $\mathcal{E}$-compact converges to $ H^{-1}_{V(\Omega )}$ as $\eps\to 0$.
\end{theorem}
\begin{proof} By using a standard partition of unity  argument, it suffices to prove that the assumptions of Theorem \ref{lemma: 6.2 improved} are satisfied by all the profile functions $g_{j,\epsilon }$, $g_{j}$ describing the boundaries of $\Omega_{\epsilon}, \Omega$, respectively, and this follows by choosing $\kappa_{\eps}=(d_{\mathcal{A}}^{(m-k)}(\Omega_{\eps}, \Omega ))^{\frac{1}{m}}$.
\end{proof}

In order to prove that the assumptions of Lemma \ref{lemma: 6.2 improved} are sharp, we now consider a the following geometric setting:\\[0.1cm] 
\textbf{(G2)} Let $\alpha \in \R$, $\alpha > 0$. Let $b \in C^{\infty}(\overline{W})$ a positive, non-constant periodic function, with periodicity cell given by $Y= ]-1/2, 1/2[^{N-1}$. Let us set
\begin{equation*}
g_\epsilon (\bar{x}) = \epsilon^\alpha b\bigg(\frac{\bar{x}}{\epsilon}\bigg), \quad\quad g(\bar{x}) = 0,
\end{equation*}
for all $\bar{x}\in W$. For simplicity, we set $g_0=g$ and for all $\epsilon \geq 0$ we consider the open sets
\begin{equation*}
\Omega_\epsilon   = \{(\bar{x}, x_N) \in \R^N : \bar{x} \in W, \, -1 < x_N < g_{\epsilon}(\bar{x}) \}
\end{equation*}
\vspace{0.1cm}

Then we have the following

\begin{theorem}
\label{thm: spectral stability polyharmonics}
Let $\Omega_\eps$, $\eps \geq 0$ be as in {\rm\textbf{(G2)}} and let $k \in \N$ satisfy $1 \leq k \leq m-1$. Let $V(\Omega_{\epsilon}) = W^{m,2}(\Omega_{\epsilon }) \cap W^{k,2}_0(\Omega_{\epsilon})$ for all $\epsilon \geq 0$. If $\alpha > m-k+\frac{1}{2}$, then $H^{-1}_{V(\Omega_\eps)}$  $\mathcal{E}$-compact converges to $  H^{-1}_{V(\Omega)}$ as $\eps \to 0$.
\end{theorem}
\begin{proof}
We aim at  applying Theorem \ref{lemma: 6.2 improved} with  $\kappa_\epsilon = \epsilon^{\alpha \theta}\norma{b}_{\infty} $, for some $\theta \in (0,1)$ to be specified. By the classical Gagliardo-Nirenberg interpolation inequality
$$
\norma{D^\beta f}_{\infty} \leq C (\sum_{|\alpha|=m} \norma{D^\alpha f}_{\infty})^{|\beta| /m}\norma{f}^{1-|\beta|/m}_{\infty},
$$
for all $f \in W^{m, \infty}(\Omega)$ (see e.g., \cite[p.125]{Nirenb}), in order to verify condition (iii) in Theorem \ref{lemma: 6.2 improved} it is sufficient to verify it for $|\beta|=0$ and $|\beta|=m$ (see also \cite[Proposition 6.17]{ArrLamb}). When $\abs{\beta} = 0$ we have
\[
\lim_{\epsilon \to 0}\frac{\norma{g_\epsilon - g}_\infty}{\kappa_\epsilon^{m-k+1/2}} = c \lim_{\epsilon \to 0} \frac{\epsilon^\alpha}{\epsilon^{\alpha\theta(m-k+1/2)}} =  c \lim_{\epsilon \to 0}  \epsilon^{\alpha (1-\theta(m-k-1/2)) },
\]
where $c$ is a constant depending only on $\norma{b}_{\infty}$. The right-hand side clearly tends to $0$ as soon as $\theta < \frac{1}{m-k+1/2}$.\\
When $\abs{\beta} = m$, we must check that
$
\lim_{\epsilon \to 0} \frac{D^\beta g_\epsilon}{\kappa_\epsilon^{-k + 1/2}} = 0.
$
Note that
\[
\left\| \frac{D^\beta g_\epsilon}{\kappa_\epsilon^{-k + 1/2}} \right\|_{\infty}= c \frac{\eps^{\alpha -m}}{\eps^{\alpha\theta(-k+1/2)}} = \eps^{\alpha(1-\theta(-k+1/2))-m},
\]
and the right hand side tends to zero if and only if
\begin{equation}
\label{proof: ineq alphatheta}
\alpha \Big( 1 + \theta \Big(k - \frac{1}{2}\Big) \Big) - m > 0.
\end{equation}
By letting $\theta \to \frac{1}{m-k+1/2}$ in \eqref{proof: ineq alphatheta} we obtain that inequality \eqref{proof: ineq alphatheta} is satisfied when $\alpha > m-k+1/2$, true by assumption. By Lemma \ref{lemma: 6.2 improved} we deduce the validity of Theorem \ref{thm: spectral stability polyharmonics}.
\end{proof}

\begin{remark} When $k = m-1$, Theorem \ref{thm: spectral stability polyharmonics} states that if $\alpha > \frac{3}{2}$, $H^{-1}_{V(\Omega_\eps)} \overset{\mathcal{C}}{\rightarrow} H^{-1}_{V(\Omega)}$ as $\eps \to 0$, independently on $m \geq 2$. Actually, it is possible to prove that $\alpha = 3/2$ in this case is the critical exponent, in the sense that when $\alpha \leq 3/2$  the operator   $H^{-1}_{V(\Omega_\eps)}$ does not converge to $H^{-1}_{V(\Omega)}$. We refer to Theorem \ref{thm: poly strong} for a complete discussion about the spectral convergence of $H_{V(\Omega_\eps)}$ depending on $\alpha$.
\end{remark}


\section{A polyharmonic Green formula}
\label{sec: poly green formula}
In this section we provide a formula which turns out to be useful in recognising the possible natural boundary conditions for polyharmonic operators of any order. Let us begin by stating an easy integration-by-parts formula.

\begin{proposition}
Let $\Omega$ be a bounded domain of class $C^{0,1}$ in $\R^N$. Let $m \in \N$ and let $f\in C^{m+1}(\overline{\Omega})$, $\varphi \in C^m(\overline{\Omega})$. Then
\begin{equation}
\label{byparts 1}
\begin{split}
\int_{\Omega} D^m f : D^m \varphi \, \diff{x} = - \int_{\Omega} D^{m-1}&(\Delta f) : D^{m-1} \varphi \, \diff{x} \\
&+ \int_{\partial \Omega} D^{m} f : (n \otimes D^{m-1}\varphi) \,\diff{S},
\end{split}
\end{equation}
where the symbol $:$ stands for the Frobenius product, $n$ is the unit outer  normal to $\partial \Omega$, and $\otimes$ is the tensor product, defined by $(n \otimes D^{m-1}\varphi)_{i,j_1, \cdots, j_{m-1}} = n_i \frac{\partial^{m-1}\varphi}{\partial x_{j_1} \cdots \partial x_{j_{m-1}}}$ for all $i,j_1, \cdots, j_{m-1} \in \{1, \cdots, N\}$.
\end{proposition}
\begin{proof}
The proof is a simple integration by parts.  Indeed, dropping the summation symbols we get
\begin{small}
\[
\begin{split}
&\int_{\Omega}D^{m} f : D^{m}\varphi \, \diff{x} = \int_{\Omega} \frac{\partial^{m}f}{\partial x_{j_1} \cdots \partial x_{j_{m}} } \frac{\partial^{m}\varphi}{\partial x_{j_1} \cdots \partial x_{j_{m}} } \, \diff{x}\\
&= - \int_{\Omega} \frac{\partial^{m+1}f}{\partial x_{j_1}^2 \cdots \partial x_{j_{m}} } \frac{\partial^{m-1}\varphi}{\partial x_{j_2} \cdots \partial x_{j_{m}} } \, \diff{x} + \int_{\partial \Omega} \frac{\partial^{m}f}{\partial x_{j_1} \cdots \partial x_{j_{m}} } \frac{\partial^{m-1}\varphi}{\partial x_{j_2} \cdots \partial x_{j_{m}} } n_{j_1} \, \diff{S} \\
&= - \int_{\Omega} D^{m-1} (\Delta f) : D^{m-1} \varphi \, \diff{x} + \int_{\partial \Omega} (D^{m} f) : ( n \otimes D^{m-1}\varphi) \,\diff{S}.
\end{split}
\]\qedhere
\end{small} 
\end{proof}

By applying $m$ times the integration by parts argument used in the proof of formula~\eqref{byparts 1}, we deduce the validity of the following

\begin{corollary}
Let $m \in \N$. Let $f\in C^{2m}(\overline{\Omega})$, $\varphi \in C^m(\overline{\Omega})$.
\begin{multline}
\label{byparts 2}
\int_{\Omega} D^m f : D^m \varphi \, \diff{x} = (-1)^m \int_{\Omega} \Delta^m f  \varphi \, \diff{x}\\
 + \sum_{k=0}^{m-1} (-1)^k \int_{\partial \Omega} (D^{m-k}(\Delta^k f)) : (n \otimes D^{m-k-1}\varphi) \,\diff{S}.
\end{multline}
\end{corollary}

\medskip


\begin{theorem}[Polyharmonic Green Formula - Flat case]
\label{PGF: flat case}
Let $H$ be the half-space $H = \{(\bar{x}, x_N) \in \R^N : x_N < 0 \}. $
Let $m \in \N$. Let $f \in C^{2m}(\overline{H})$, $\varphi \in C^m(\overline{H})$ with compact support in $ \overline{H}$. Then,
\begin{equation}
\label{polyharmonic Green formula}
\int_H D^m f : D^m \varphi \, \diff{x} = (-1)^m \int_{H} \Delta^m f  \varphi \, \diff{x} + \sum_{t=0}^{m-1} \int_{\R^{N-1}} B_t(f) \frac{\partial^t\varphi}{\partial x_N^t} \diff{\bar{x}},
\end{equation}
where $B_t : C^{2m}(\partial H) \to C^{t+1}(\partial H)$ is defined by
\begin{equation}
\begin{split}
B_t(f) &= \sum_{l=t}^{m-1} (-1)^{m-t-1} \binom{l}{t} \Delta_{N-1}^{l-t} \Bigg(\frac{\partial^{t+1}}{\partial x_N^{t+1}}(\Delta^{m-l-1}f) \Bigg),\\
\end{split}
\end{equation}
and $\Delta_{N-1}$ is the Laplace operator in the first $N-1$ variables.
\end{theorem}
\begin{proof}
Let $r = m-k-1$. First note that  we can write
\begin{small}
\begin{multline}
\label{chain}
\int_{\R^{N-1}} \bigg(D^{r}\bigg(\Delta^k\bigg( \frac{\partial f}{\partial x_N}\bigg)\bigg)\bigg) : D^{r}\varphi \,\diff{\bar{x}}\\
= \sum_{t=0}^{r} \binom{r}{t} \int_{\R^{N-1}} \bigg(D_{\bar{x}}^{r-t} \bigg(\Delta^k \bigg(\frac{\partial^{t+1}f}{\partial x^{t+1}_N} \bigg)\bigg)\bigg) : \bigg(D_{\bar{x}}^{r-t} \bigg( \frac{\partial^t \varphi}{\partial x^t_N} \bigg)\bigg) \, d\bar{x}.
\end{multline}
\end{small}
Then, by using \eqref{chain} in the last integral in the right-hand side of~\eqref{byparts 2} we get the following as boundary term
\begin{equation}
\label{integral 1}
\sum_{k=0}^{m-1} (-1)^k\sum_{t= 0}^{r} \binom{r}{t} \int_{\R^{N-1}} D_{\bar{x}}^{r-t}\Bigg(\frac{\partial^{t+1}(\Delta^k f)}{\partial x_N^{t+1}}\Bigg) : D_{\bar{x}}^{r-t}\Bigg(\frac{\partial^t\varphi}{\partial x_N^t}\Bigg)\, \diff{\bar{x}}.
\end{equation}
By dropping the summation symbols, the integrand in  \eqref{integral 1} becomes
\begin{equation}
\label{integral 2}
\int_{\R^{N-1}} \frac{\partial^{r-t}}{\partial x_{i_1} \cdots \partial x_{i_{r-t}}} \Bigg(\frac{\partial^{t+1}(\Delta^k f)}{\partial x_N^{t+1}}\Bigg)  \frac{\partial^{r-t}}{\partial x_{i_1} \cdots \partial x_{i_{r-t}}} \Bigg(\frac{\partial^t\varphi}{\partial x_N^t}\Bigg)\, \diff{\bar{x}},
\end{equation}
where the indexes $i_j$ run on the first $N-1$ coordinates. By integrating by parts $r-t$ times in $i_1,\dots,i_{r-t}$ in \eqref{integral 2} we deduce that \eqref{integral 1} equals
\[
\sum_{k=0}^{m-1} (-1)^{m-t-1} \sum_{t=0}^{r} \binom{r}{t} \int_{\R^{N-1}} \frac{\partial^{2(r-t)}}{\partial^2x_{i_1} \cdots \partial^2x_{i_{r-t}}} \Bigg(\frac{\partial^{t+1}(\Delta^k f)}{\partial x_N^{t+1}}\Bigg)\, \frac{\partial^t\varphi}{\partial x_N^t}\, \diff{\bar{x}},
\]
where we have no other boundary terms because $\varphi$ has compact support. We rewrite the last expression as
\begin{equation}
\label{integral 3}
\sum_{k=0}^{m-1} (-1)^{m-t-1}\sum_{t= 0}^{r} \binom{r}{t} \int_{\R^{N-1}}  \Delta_{N-1}^{r-t}\Bigg(\frac{\partial^{t+1}(\Delta^k f)}{\partial x_N^{t+1}}\Bigg)\, \frac{\partial^t\varphi}{\partial x_N^t} \diff{\bar{x}}.
\end{equation}
We now apply the change of summation index $r=m-k-1$ in the first sum of \eqref{integral 3}. We deduce that \eqref{integral 3} equals
\begin{equation}
\label{integral 4}
\sum_{r=0}^{m-1} (-1)^{m-t-1} \sum_{t=0}^r \binom{r}{t} \int_{\R^{N-1}} \Delta_{N-1}^{r-t}\Bigg(\frac{\partial^{t+1}(\Delta^{m-r-1} f)}{\partial x_N^{t+1}}\Bigg)\, \frac{\partial^t\varphi}{\partial x_N^t} \diff{\bar{x}}.
\end{equation}
By exchanging the two sums in \eqref{integral 4} we get \eqref{polyharmonic Green formula}.
\end{proof}

\begin{remark} If $m=2$, then \eqref{polyharmonic Green formula} reads
\begin{multline*}
\int_H D^2 f : D^2\varphi \, \diff{x} =  \int_H \Delta^2f \varphi\,\diff{x} + \int_{\R^{N-1}} \frac{\partial^{2}f}{\partial x_N^{2}} \frac{\partial \varphi}{\partial x_N}\, \diff{\bar{x}} \\
-\!\! \int_{\R^{N-1}} \Bigg( \Delta_{N-1}\Bigg( \frac{\partial f}{\partial x_N}\Bigg) + \Delta\Bigg( \frac{\partial f}{\partial x_N}\Bigg) \Bigg) \varphi\,\diff{\bar{x}},
\end{multline*}
which is consistent with the formula provided in \cite[Lemma 8.56]{ArrLamb}. Indeed, if the domain is a hyperplane, the boundary integral
$
\int_{\partial H} ( \Div_{\partial H} (D^2f \cdot n)_{\partial \Omega})\, \varphi\,\diff{S}
$
appearing in \cite[Lemma 8.56]{ArrLamb} coincides with
$
\int_{\R^{N-1}} \Delta_{N-1}( \frac{\partial f}{\partial x_N} )\, \varphi \,\diff{\bar{x}}.
$
\end{remark}

\begin{theorem}\label{polygreenthm}
Let $\Omega$ be a bounded domain of $\R^N$ of class $C^{0,1}$, $m \in \N$, $m \geq 2$. Let $f \in W^{2m,2}(\Omega) \cap W^{m-1,2}_0(\Omega)$ and $\varphi \in W^{m,2}(\Omega) \cap W^{m-1,2}_0(\Omega)$. Then
\begin{equation}
\label{polyharmonic Green strong BC}
\int_{\Omega} D^m f : D^m \varphi\, dx = (-1)^m \int_{\Omega} \Delta^m f \varphi dx + \int_{\partial \Omega} \frac{\partial^m f}{\partial n^m} \frac{\partial^{m-1} \varphi}{\partial n^{m-1}}\, dS.
\end{equation}
\end{theorem}
\begin{proof}
By \eqref{byparts 2} it is easy to see that
\begin{small}
\begin{equation}
\label{rmk: easy green}
\int_{\Omega} D^m f : D^m \varphi\, dx = (-1)^m \int_{\Omega} \Delta^m f \varphi\, dx + \int_{\partial \Omega} D^m f : (n \otimes D^{m-1} \varphi) \, dS,
\end{equation}
\end{small}
for all $\varphi \in W^{m,2}(\Omega) \cap W^{m-1,2}_0(\Omega)$, since $D^l \varphi = 0$ on $\partial \Omega $ for all $l \leq m-2$. We note that
$
D^m f : (n \otimes D^{m-1} \varphi) = (n^T D^m f) : D^{m-1} \varphi.
$
Moreover we claim that $D^{m-1} \varphi  = \frac{\partial^{m-1} \varphi}{\partial n^{m-1}} \bigotimes^{m-1}_{i=1} n$ on $\partial \Omega$ and we prove it by induction.
If $m=2$ the claim is a direct consequence of the gradient decomposition $\nabla|_{\partial \Omega} = \nabla_{\partial \Omega} + \frac{\partial}{\partial n} n$. Now we assume that $m > 2$ and that the claim holds for $m-1$. Then, by using the fact that
$
D^{m-2} \varphi |_{\partial \Omega} = 0,
$
for all $\varphi \in W^{m,2}(\Omega) \cap W^{m-1,2}_0(\Omega)$, we get
\[
\begin{split}
D^{m-1} \varphi |_{\partial \Omega} = D ( D^{m-2} \varphi)|_{\partial \Omega} &= \bigg(D\bigg(\frac{\partial^{m-2} \varphi}{\partial n^{m-2}} \bigotimes^{m-2}_{i=1} n \bigg)n\bigg) \otimes n = \frac{\partial^{m-1} \varphi}{\partial n^{m-1}} \bigotimes^{m-1}_{i=1} n,
\end{split}
\]
for all $\varphi \in W^{m,2}(\Omega) \cap W^{m-1,2}_0(\Omega)$. This proves the claim. Then we can rewrite \eqref{rmk: easy green} as
\begin{multline}
\int_{\Omega} D^m f : D^m \varphi \, dx = (-1)^m \int_{\Omega} \Delta^m f \varphi \, dx + \int_{\partial \Omega} \frac{\partial^{m-1} \varphi}{\partial n^{m-1}}(n^T D^m f) : \bigg( \bigotimes^{m-1}_{i=1} n\bigg) \, dS,
\end{multline}
and since $(n^T D^m f) : \Big( \bigotimes^{m-1}_{i=1} n\Big) = D^m f : \Big( \bigotimes^{m}_{i=1} n \Big )= \frac{\partial^m f}{\partial n^m}$ we deduce \eqref{polyharmonic Green strong BC}.
\end{proof}


\section[Strong intermediate boundary conditions]{Polyharmonic operators with strong intermediate boundary conditions}
\label{sec: polyharmonic operators}
Let $\Omega_\eps$, $\eps \geq 0$ be as in \textbf{(G2)}. Consider the polyharmonic operators $(-\Delta)^m + \mathbb{I}$ subject to strong intermediate boundary conditions, corresponding to the energy space $V(\Omega_\eps) := W^{m,2}(\Omega_\eps) \cap W_0^{m-1,2}(\Omega_\eps)$. More precisely, let $H_{\Omega_\epsilon,S}$ be the non-negative self-adjoint operator such that
\begin{equation}
\label{polyharmonic H}
(H_{\Omega_\epsilon,S}u, v)_{L^2(\Omega_\eps)} = (H^{1/2}_{\Omega_\epsilon,S}u,\, H^{1/2}_{\Omega_\epsilon,S}v)_{L^2(\Omega_\eps)} = Q_{\Omega_\eps} (u,v),
\end{equation}
for all functions $u,v \in W^{m,2}(\Omega_\eps) \cap W_0^{m-1,2}(\Omega_\eps)$, where
$
Q_{\Omega_\eps} (u,v) := \int_{\Omega_\epsilon} D^m\! u : D^m\! v + uv\, dx,
$
is the quadratic form canonically associated with $H_{\Omega_\epsilon,S}$. As it is explained in Section 2  the equation
$
H_{\Omega_\epsilon,S}u=f
$
with  datum $f\in L^2(\Omega_{\epsilon})$, corresponds exactly to the weak Poisson problem \eqref{intro: mainweak}.

Let $H_{\Omega, D}$ be  the polyharmonic operator satisfying strong intermediate boundary conditions on $\partial \Omega \setminus \overline{W}$ and Dirichlet boundary conditions on $W$, whose associated boundary value problem reads
\begin{equation}
\label{eq: Dirichlet BC poly}
\begin{cases}
(-\Delta)^m u + u = f, &\textup{in $\Omega_\eps$,}\\
\frac{\partial^l u}{\partial n^l} = 0, &\textup{on $W$, for all $0 \leq l \leq m-1$,}\\
\frac{\partial^l u}{\partial n^l} = 0, &\textup{on $\partial \Omega_\eps\setminus \overline{W}$, for all $0 \leq l \leq m-2$,}\\
\frac{\partial^m u}{\partial n^m} = 0, &\textup{on $\partial \Omega_\eps \setminus \overline{W}$.}
\end{cases}
\end{equation}
Note that we are identifying  $W$ with $W \times \{0\}$. Then the following theorem holds.

\begin{theorem}
\label{thm: poly strong}
Let $m \in \N$, $m\geq 2$, $\Omega_\eps$ as in {\rm \textbf{(G2)}}, $H_{\Omega_\eps}$ as in \eqref{polyharmonic H}, for all $\eps > 0$. Then the following statements hold true.
\begin{enumerate}[label=(\roman*)]
\item \textup{[Spectral stability]} If $\alpha > 3/2$, then $H^{-1}_{\Omega_\epsilon,S} \overset{\mathcal{C}}{\rightarrow} H^{-1}_{\Omega, S}$ as $\eps \to 0$.
\item \textup{[Instability]} If $\alpha < 3/2$, then $H^{-1}_{\Omega_\epsilon,S} \overset{\mathcal{C}}{\rightarrow} H^{-1}_{\Omega, D}$ as $\eps \to 0$, where $H_{\Omega, D}$ is defined in \eqref{eq: Dirichlet BC poly}.
\item \textup{[Strange term]} If $\alpha = 3/2$, then $H^{-1}_{\Omega_\epsilon,I} \overset{\mathcal{C}}{\rightarrow} \hat{H}^{-1}_{\Omega}$ as $\eps \to 0$, where $\hat{H}_\Omega$ is the operator $(-\Delta)^m + \mathbb{I}$ with strong intermediate boundary conditions on $\partial \Omega \setminus \overline{W}$ and the following boundary conditions on $W$: $D^l u=0$, for all $l \leq m-2$, $\partial^m_{x_N} u + K \partial_{x_N}^{m-1} u = 0$, where the factor $K$ is given by
{\[\begin{split}
    K = \int_{Y\times (-\infty,0)} |D^m V|^2\,\diff{y} = - \int_{Y} \Bigg(\frac{\partial^{m-1} (\Delta V)}{\partial x_N^{m-1}} + (m-1) \Delta_{N-1} \bigg( \frac{\partial^{m-1} V}{\partial x_N^{m-1}} \bigg)   \Bigg)  b(\bar{y}) \diff{\bar{y}},
\end{split}\]}and the function $V$ is $Y$-periodic in the variable $\bar{y}$ and satisfies the following microscopic problem
    \[
    \begin{cases}
    (-\Delta)^m V = 0, &\textup{in $Y \times (-\infty, 0)$},\\
    \frac{\partial^l V}{\partial n^l}(\bar{y}, 0) = 0, &\textup{on $Y$, for all $0 \leq l \leq m-3$},\\
    \frac{\partial^{m-2} V}{\partial y^{m-2}_N}(\bar{y}, 0) = b(\bar{y}), &\textup{on $Y$},\\
    \frac{\partial^m V}{\partial y_N^m}(\bar{y}, 0) = 0, &\textup{on $Y$}.
    \end{cases}
    \]
\end{enumerate}
\end{theorem}
\begin{proof}
Statement $(i)$ is a straightforward application of  Theorem \ref{thm: spectral stability polyharmonics} with $k=m-1$. To prove $(ii)$ we check that Condition (C) in Definition \ref{def: cond C} is satisfied with $V(\Omega) = W^{m,2}_{0,W}(\Omega)\cap W^{m-1,2}_0(\Omega)$, and $V(\Omega_\eps) = W^{m,2}(\Omega_\eps) \cap W^{m-1,2}_0(\Omega_\eps)$. Here $W^{m,2}_{0,W}(\Omega)$ is the closure in $W^{m,2}(\Omega)$ of the space of functions vanishing in a neighborhood of $W$. Let $K_\eps = \Omega$ for all $\eps > 0$.  Then we see immediately that condition \eqref{def: condition C null measure} and condition (C1) are satisfied. We define now $T_\eps$ as the extension by zero operator from $W^{m,2}_{0,W}(\Omega)$ to $W^{m,2}(W \times (-1, +\infty))$ and $E_\eps$ as the restriction operator to $\Omega$. With these definitions it is not difficult to prove that conditions (C2) and (C3)(i),(ii) are satisfied. It remains to prove that condition (C3)(iii) holds. Let $v_\eps \in W^{m,2}(\Omega_\eps) \cap W^{m-1,2}_0(\Omega_\eps)$ be such that $\norma{v_\eps}_{W^{m,2}(\Omega_\eps)} \leq C$ for all $\eps > 0$. Possibly passing to a subsequence there exists a function $v \in W^{m-1,2}(\Omega)$ such that $v_\eps|_{\Omega} \rightharpoonup v$ in $W^{m,2}(\Omega)$ and $v_\eps|_{\Omega} \to v$ in $W^{m-1,2}(\Omega)$. By considering the sequence of functions $T_\eps (v_\eps|\Omega)$ it is not difficult to prove that $v \in W^{m-1,2}_0(\Omega)$. It remains to check that $\frac{\partial^{m-1} v}{\partial x^{m-1}_N} = 0$ on $W \times \{0\}$. This is proven exactly as in~\cite[Theorem 7.3]{ArrLamb} by applying Lemma 4.3 from \cite{CasDiaz} to the vector field $V_\eps^i$ defined by
\[
V_\eps^i = \bigg(0, \cdots, 0, - \frac{\partial^{m-1} v_\eps}{\partial x_N^{m-1}}, 0, \cdots, 0, \frac{\partial^{m-1} v_\eps}{\partial x_N^{m-2} \partial x_i} \bigg),
\]
for all $i = 1, \dotsc, N-1$, where the only non-zero entries are the $i$-th and the $N$-th ones. We remark that it is possible to apply Lemma 4.3 from \cite{CasDiaz} because by Theorem \ref{thm: spectral stability polyharmonics} the critical threshold for all the polyharmonics operator with strong intermediate boundary conditions is $\alpha = 3/2$, which coincides with the critical value in \cite{CasDiaz}. We then deduce that
$
\frac{\partial^{m-1} v (\bar{x},0)}{\partial x_N^{m-1}} \frac{\partial b(\bar{y})}{\partial y_i} =  0$,  a.e. $ W\times Y$.
Since $b$ is a non-constant smooth function  we must have  $\frac{\partial^{m-1} v (\bar{x},0)}{\partial x_N^{m-1}} =0$ a.e. on $W$. This concludes the proof of condition $(C3)(iii)$.

We provide a proof of $(iii)$ in  Sections 5.1 and 5.2.
\end{proof}

\begin{remark}\label{misprint} We take the chance to point out a misprint in \cite[Theorem~1,~(ii)]{ArrFerLambtrih} where the  condition $\partial^m_{x_N} u + K \partial_{x_N}^{m-1} u = 0$ in our Theorem~\ref{thm: poly strong} (iii) above, appears  for $m=3$ with  $-K$ instead of $+K$ as it should be.
\end{remark}


\subsection{Critical case - Macroscopic problem.}

In this section we prove Theore~\ref{thm: poly strong} (iii). Let us define a diffeomorphism $\Phi_\epsilon$ from $\Omega_\eps$ to $\Omega$ by
\[
\Phi_\eps( \bar{x}, x_N) = (\bar{x}, x_N - h_\eps(\bar{x}, x_N)), \quad \textup{for all $x=(\bar{x}, x_N) \in \Omega_\eps$,}
\]
where $h_\epsilon$ is defined by
\[
h_\epsilon (\bar{x}, x_N) =
\begin{cases}
0, &\textup{if $-1 \leq x_N \leq -\epsilon$},\\
g_\epsilon(\bar{x})\Big(\frac{x_N+\epsilon}{g_\epsilon(\bar{x})+\epsilon}\Big)^{m+1}, &\textup{if $-\epsilon \leq x_N \leq g_\epsilon(\bar{x})$}.
\end{cases}
\]
By standard calculus one can prove the following

\begin{lemma}
\label{lemma: h_eps}
The map $\Phi_\epsilon$ is a diffeomorphism of class $C^m$ and there exists a constant
$c > 0$ independent of $\epsilon$ such that
$
\abs{h_\epsilon} \leq c \epsilon^\alpha$ and $\left \lvert D^l h_\epsilon \right \rvert \leq c \epsilon^{\alpha - l}$,
for all $l= 1, \dots, m$, $\epsilon > 0$ sufficiently small.
\end{lemma}

As in \cite[Section 8.1]{ArrLamb}, we introduce the pullback operator $T_\epsilon$ from $L^2(\Omega)$ to $L^2(\Omega_\epsilon)$ given by $T_\epsilon u = u \circ \Phi_\epsilon$ for all $u\in L^2(\Omega)$. 

In order to proceed we find convenient to recall some notation and results in homogenization theory regarding the unfolding operator. We refer to \cite{Allaire, CioDo, CioDamGri, Daml} for the proof of the main properties of the operator, and we mention that recent developments can be found in the article \cite{ArrVillPesSIAM}.\\
For any $k \in \Z^{N-1}$ and $\eps > 0$ we define
\begin{equation}
\label{def: anistropic unfolding cell}
\left\{
\begin{aligned}
&C^k_\eps = \eps k + \eps Y,\\
&I_{W,\eps} = \{k \in \Z^{N-1} : C^k_\eps \subset W\},\\
&\widehat{W}_\eps = \bigcup_{k \in I_{W, \eps}} C^k_\eps.
\end{aligned}
\right.
\end{equation}
Then we give the following

\begin{definition}
\label{def: anisotropic unfolding}
Let $u$ be a real-valued function defined in $\Omega$. For any $\epsilon > 0$ sufficiently small the unfolding $\hat{u}$ of $u$ is the real-valued function defined on $\widehat{W}_\epsilon \times Y \times (-1/\epsilon, 0)$ by
\[
\hat{u}(\bar{x}, \bar{y}, y_N) = u\Big( \epsilon\Big[\frac{\bar{x}}{\epsilon}\Big] + \epsilon \bar{y}, \epsilon y_N \Big),
\]
for almost all $(\bar{x}, \bar{y}, y_N)) \in \widehat{W}_\epsilon \times Y \times (-1/\epsilon, 0)$, where $\big[\frac{\bar{x}}{\epsilon}\big]$ denotes the integer part of the vector $\bar{x} \epsilon^{-1}$ with respect to $Y$, i.e., $[\bar{x} \epsilon^{-1}] = k$ if and only if $\bar{x} \in C^k_\epsilon$.
\end{definition}

The following lemma will be often used in the sequel. For a proof we refer to  \cite[Proposition 2.5(i)]{CioDamGri2}.

\begin{lemma}
\label{lemma: exact int formula}
Let $a \in [-1,0[$ be fixed. Then
\begin{equation}
\label{eq: exact int formula}
\int_{\widehat{W}_\eps \times (a,0)} u(x) dx = \eps \int_{\widehat{W}_\eps \times Y \times (a/\eps,0)} \hat{u}(\bar{x},y) d\bar{x} dy
\end{equation}
for all $u \in L^1(\Omega)$ and $\eps>0$ sufficiently small. Moreover
\begin{equation*}
\int_{\widehat{W}_\eps \times (a,0)} \left|\frac{\partial^l u(x)}{\partial x_{i_1} \cdots \partial x_{i_l}} \right|^2 dx = \eps^{1-2l} \int_{\widehat{W}_\eps \times Y \times (a/\eps,0)} \left| \frac{\partial^l \hat{u}}{\partial y_{i_1} \cdots \partial y_{i_l}}(\bar{x},y) d\bar{x} \right|^2dy,
\end{equation*}
for all $l \leq m$,  $u \in W^{m,2}(\Omega)$ and $\eps>0$ sufficiently small.
\end{lemma}

Let $W^{m,2}_{{\rm Per}_Y, {\rm loc}}(Y \times (-\infty,0))$ be the subspace of $W^{m,2}_{\rm loc}(\R^{N-1} \times (-\infty,0))$ containing $Y$-periodic functions in the first $(N-1)$ variables $\bar{y}$. We then define $W^{m,2}_{loc}(Y \times (-\infty,0))$ to be the space of functions in
$W^{m,2}_{{\rm Per}_Y, {\rm loc}}(Y \times (-\infty,0))$ restricted to $Y \times (-\infty,0)$. Finally we set
\begin{multline}
w^{m,2}_{{\rm Per}_Y}(Y \times (-\infty, 0)) := \big\{ u \in W^{m,2}_{{\rm Per}_Y, {\rm loc}}(Y \times (-\infty,0))\\ 
: \norma{D^\gamma u}_{L^2(Y \times (-\infty,0))} < \infty, \forall |\gamma| = m \big\}.
\end{multline}
For any $d<0$, let $\mathcal{P}_{hom,y}^{l}(Y \times (d,0))$ be the space of homogeneous polynomials of degree at most $l$ restricted to the  domain $(Y \times (d,0))$. Let $\eps > 0$ be fixed. We define the projectors $P_i$ from $L^2(\widehat{W}_\eps, W^{m,2}(Y \times (-1/\eps , 0)))$ to $L^2(\widehat{W}_\eps, \mathcal{P}_{hom,y}^{i}(-1/\eps, 0))$ by setting
\begin{equation*}
P_i (\psi )= \sum_{|\eta|=i} \int_{Y} D^\eta \psi(\bar{x},\bar{\zeta},0) d\bar{\zeta} \frac{y^\eta}{\eta!}
\end{equation*}
for all $i = 0, \dots, m-1$. We now set $Q_{m-1} = P_{m-1}$, $Q_{m-2} = P_{m-2} (\mathbb{I} - Q_{m-1})$, etc., up to $Q_{0} = P_0\big(\mathbb{I} - \sum_{j=1}^{m-1} Q_j\big)$. Note that $Q_{m-j}$, $j=1, \dots, m$ is a projection on the space of homogeneous  polynomials of degree $m-j$, with the property that $Q_{m-k} (p) = 0$ for all polynomials $p$ of degree  $m-k$ with $k\ne j$. We finally set
\begin{equation}
\label{def: p}
\p= Q_0 + Q_1 + \cdots + Q_{m-1},
\end{equation}
which is a projector on the space of polynomials in $y$ of degree at most $m-1$. Note that $D_y^\beta \p(\psi )(\bar{x}, \bar{y}, 0) = \int_{Y}D_y^\beta\psi (\bar{x}, \bar{y}, 0) d\bar y$ for all  $|\beta| = 0,\dots, m-1$. 
In particular, it follows that
$ \int_{Y} ( D_y^\beta\psi (\bar{x}, \bar{y}, 0) - D_y^\beta \p(\psi )(\bar{x}, \bar{y}, 0) ) d\bar y = 0
$
for almost all $\bar{x}$ in $\widehat{W}_\eps$, for all  $|\beta| = 0,\dots, m-1$. 

\begin{lemma}
\label{lemma: unfolding convergence poly}
Let $m \in \N$, $m \geq 2$ be fixed. The following statements hold:
\begin{enumerate}[label=(\roman*)]
\item Let $v_\epsilon \in W^{m,2}(\Omega)$ with $\norma{\hat {v_\epsilon}}_{W^{m,2}(\Omega)} \leq M$, for all $\epsilon > 0$. Let $V_\epsilon$ be defined by
\[
\begin{split}
V_\epsilon(\bar{x}, y) =  & \hat{v_\epsilon}(\bar{x}, y) - \p(v_\epsilon)(\bar{x}, y),
\end{split}
\]
for $(\bar{x}, y) \in \widehat{W_\epsilon}\times Y\times (-1/\epsilon, 0)$, where $\p$ is defined by \eqref{def: p} . Then there exists a function $\hat{v}\in L^2(W, w^{m,2}_{\textup{Per}_Y}(Y\times (-\infty,0)))$ such that, possibly passing to a subsequence, for every $d<0$
\begin{enumerate}[label=(\alph*)]
\item $\frac{D_y^{\gamma}V_\epsilon}{\epsilon^{m-1/2}} \rightharpoonup D_y^{\gamma}\hat{v}$ in $L^2(W\times Y \times (d,0))$ as $\epsilon \to 0$, for any $\gamma \in \N_0^N$, $\abs{\gamma} \leq m-1$.
\item $\frac{D_y^{\gamma}V_\epsilon}{\epsilon^{m-1/2}} \rightharpoonup D_y^{\gamma}\hat{v}$ in $L^2(W\times Y \times (-\infty,0))$ as $\epsilon \to 0$, for any $\gamma \in \N_0^N$, $\abs{\gamma} = m$,
\end{enumerate}
where it is understood that the functions $V_\epsilon, D_y^{\gamma}V_\epsilon$ are extended by zero to the whole of $W\times Y \times (-\infty,0)$ outside their natural domain of definition $\widehat{W_\epsilon}\times Y\times (-1/\epsilon, 0) $.
\item If $\psi\in W^{1,2}(\Omega)$, then $\lim_{\eps \to 0} \widehat{(T_\epsilon \psi)_{|\Omega}} =  \psi(\bar{x},0)$ in $L^2(W \times Y \times (-1,0))$.
\end{enumerate}
\end{lemma}
\begin{proof}
The proof follows as in the proof ~\cite[Lemma 8.9]{ArrLamb} by noting that $\p$ is a projector on the space of polynomials of degree at most $m-1$, so that a Poincar\'{e}-Wirtinger-type inequality still holds. 
\end{proof}

Let $f_\epsilon\in L^2(\Omega_\epsilon)$ and $f\in L^2(\Omega)$ be such that $f_\epsilon \rightharpoonup f$ in $L^2(\R^N)$ as $\epsilon \to 0$,
with the understanding that the functions are extended by zero outside their natural domains.
Let $v_\epsilon \in V(\Omega_\epsilon) = W^{m,2}(\Omega_\epsilon) \cap W^{m-1,2}_0(\Omega_\epsilon)$ be such that  for all $\epsilon >0$ small enough
\begin{equation}
\label{eq: Poisson prblm poly}
H_{\Omega_\epsilon, S} v_\epsilon = f_\epsilon .
\end{equation}
Then $\norma{v_\epsilon}_{W^{m,2}(\Omega_\epsilon)} \leq M$ for all $\epsilon > 0$ sufficiently small, hence, possibly passing to a subsequence there exists $v \in W^{m,2}(\Omega)\cap W^{m-1,2}_0(\Omega)$ such that $v_\epsilon \rightharpoonup v$ in $W^{m,2}(\Omega)$ and $v_\epsilon \rightarrow v$ in $L^2(\R^N)$.

Let $\varphi \in V(\Omega) = W^{m,2}(\Omega) \cap W^{m-1,2}_0(\Omega)$ be  fixed. Since $T_\epsilon \varphi \in V(\Omega_\epsilon)$, by \eqref{eq: Poisson prblm poly} we have
\begin{equation}
\label{eq: Poisson 1 poly}
\int_{\Omega_\epsilon} D^mv_\epsilon : D^m T_\epsilon \varphi \, \diff{x} + \int_{\Omega_\epsilon}v_\epsilon T_\epsilon \varphi\, \diff{x} = \int_{\Omega_\epsilon} f_\epsilon T_\epsilon \varphi\, \diff{x},
\end{equation}
and passing to the limit as $\epsilon \to 0$ we get
$
\int_{\Omega_\epsilon} v_\epsilon T_\epsilon \varphi\, \diff{x} \to \int_{\Omega} v \varphi \,\diff{x}$ and  $ \int_{\Omega_\epsilon} f_\epsilon T_\epsilon \varphi\, \diff{x} \to \int_{\Omega} f \varphi \,\diff{x}$.

Now consider the first integral in the right hand-side of \eqref{eq: Poisson 1 poly}. Set $K_\epsilon = W \times (-1, -\epsilon)$. By splitting the integral in three terms corresponding to $\Omega_\epsilon \setminus \Omega$, $\Omega \setminus K_\epsilon$ and $K_\epsilon$ and by arguing as in \cite[Section 8.3]{ArrLamb} one can show that
$
\int_{K_\epsilon} D^m v_\epsilon : D^m \varphi \, \diff{x} \to \int_\Omega D^m v : D^m \varphi \, \diff{x}$ and  $ \int_{\Omega_\epsilon \setminus \Omega} D^m v_\epsilon : D^m T_\epsilon \varphi \, \diff{x} \to 0$,
as $\epsilon \to 0$. Let us define $Q_\eps$ by
\[Q_\epsilon = \widehat{W_\epsilon} \times (-\epsilon,0).\]
We split again the remaining integral in two summands as follows:
\begin{multline}\label{eq: integralsplit poly}
\int_{\Omega_\epsilon \setminus K_\epsilon} D^m v_\epsilon : D^m T_\epsilon \varphi \, \diff{x}\\
= \int_{\Omega_\epsilon \setminus (K_\epsilon \cup Q_\epsilon)} D^m v_\epsilon : D^m T_\epsilon \varphi \, \diff{x} + \int_{Q_\epsilon} D^m v_\epsilon : D^m T_\epsilon \varphi \, \diff{x}.
\end{multline}
As in \cite[Section 8.3]{ArrLamb},
$
\int_{\Omega_\epsilon \setminus (K_\epsilon \cup Q_\epsilon)} D^m v_\epsilon : D^m T_\epsilon \varphi \, \diff{x} \to 0,
$
as $\epsilon \to 0$. It remains to analyse the limit as $\epsilon \to 0$ of the last summand in the right-hand side of \eqref{eq: integralsplit poly}. To do so, we also need the following lemma in the proof of which we  use notation and rules of calculus recalled in Section~\ref{sec: high variational}.

\begin{lemma}
\label{lemma: convergence b poly} Let $l\in \mathbb{N}$, $l\le m$,  and let $i_1,\dots, i_l\in \{1, \dots , N \}$.
 The functions $\hat{h}_\epsilon(\bar{x},y)$, $\widehat{\frac{\partial^l h_\epsilon}{\partial x_{i_1} \cdots \partial x_{i_l}}}(\bar{x},y)$
 defined for $y \in Y \times (-1,0)$,  are independent of $\bar{x}$. Moreover, $\norma{\hat{h}_\epsilon}_{L^\infty} = O(\epsilon^{3/2})$, $\norma*{\widehat{\frac{\partial^l h_\epsilon}{\partial x_{i_1} \cdots \partial x_{i_l}}}(\bar{x},y)}_{L^\infty} = O(\epsilon^{3/2 - l})$ as $\epsilon \to 0$, and if $l\geq 2$ we have $\epsilon^{l-3/2} \widehat{\frac{\partial^l h_\epsilon}{\partial x_{i_1} \cdots \partial x_{i_l}}} (\bar{x},y) \to \frac{\partial^l (b(\bar{y})(y_N + 1)^{m+1})}{\partial y_{i_1} \cdots \partial y_{i_l}}$ as $\epsilon \to 0$, uniformly in $y \in Y \times (-1,0)$.
\end{lemma}
\begin{proof}
First, note that the part of the statement involving the asymptotic behaviour of $\widehat{h}_\eps$ as $\eps \to 0$ follows directly from Lemma \ref{lemma: h_eps} and Definition \ref{def: anisotropic unfolding}. Assume now that  $l\geq 2$.
 By applying formula \eqref{eq: Leibnitz}  we have that
\begin{small}
\begin{equation}
\label{proof: contih11}
\widehat{\frac{\partial^l h_\eps}{\partial x_{i_1}\cdots \partial x_{i_l}}}(\bar{x},y) = \sum_{S \in \p(l)} \frac{\eps^{\alpha}}{\eps^{|S|}} \frac{\partial^{|S|} b(\bar{y})}{\prod_{j \in S} \partial y_{i_j}} \widehat{\frac{\partial^{l-|S|}}{\prod_{j \notin S} \partial x_{i_j}}} \Bigg(\frac{x_N+\eps}{g_\eps(\bar{x}) + \eps} \Bigg)^{m+1}
\end{equation}
\end{small}

Standard Calculus computations based on Formulas \eqref{eq: Faa di Bruno} and \eqref{eq: Leibnitz} give 
\begin{small}
\begin{multline}
\label{proof: formula derivatives x_N+eps}
\widehat{\frac{\partial^{l-|S|}}{\prod_{j \notin S} \partial x_{i_j}}}\bigg( \frac{x_N+\eps}{g_\eps(\bar{x}) + \eps} \bigg)^{m+1}\!= C\big(|S|\big)\, \eps^{-l + |S|}\frac{(y_N + 1)^{m+1-l + |S|}}{(\eps^{\alpha-1}b(\bar{y}) +1)^{m+1}} \prod_{j \notin S} \delta_{{i_j}N}\\
+\sum_{\substack{\Lambda \in \p(S^C) \\ \Lambda \neq \emptyset}} \sum_{\pi \in   {\rm Part} (\Lambda )  }
\eps^{\alpha |\pi| - |\pi| - l + |S|}  (-1)^{|\pi|} \frac{(m+ |\pi|)!}{m!} \frac{(m+1)!}{(m+1-l+|S|+|\Lambda|)!}\\
\cdot \frac{(y_N +1)^{m+1-l + |S| +|\Lambda|}}{(\eps^{\alpha-1}b(\bar y) + 1)^{m+1+|\pi|}} \prod_{k \in (S^C \setminus \Lambda)} \delta_{i_k N} \prod_{B \in \pi} \frac{\partial^{|B|}b(\bar{y})}{\prod_{l \in B} \partial y_{i_l}}.
\end{multline}
\end{small}
where $C\big(|S|\big) = \frac{(m+1)!}{(m+1-l + |S|)!} $.
By \eqref{proof: contih11} and \eqref{proof: formula derivatives x_N+eps} we deduce that
{\small
\begin{equation}
\label{proof: contih final}
\begin{split}
&\epsilon^{l-\alpha} \widehat{\frac{\partial^l h_\epsilon}{\partial x_{i_1} \cdots \partial x_{i_l}}} (\bar{x},y)\\
&= \epsilon^{l-\alpha} \sum_{S \in \p(l)} \eps^{\alpha - |S|} \frac{\partial^{|S|} b(\bar{y})}{\prod_{j \in S} \partial y_{i_j}} C\big(|S|\big) \eps^{-l + |S|}\frac{(y_N + 1)^{m+1-l + |S|}}{(\eps^{\alpha-1}b(\bar{y}) +1)^{m+1}} \prod_{j \notin S} \delta_{{i_j}N}\\
&+\eps^{l-\alpha}\!\sum_{S \in \p(l)} \eps^{\alpha - |S|} \frac{\partial^{|S|} b(\bar{y})}{\prod_{j \in S} \partial y_{i_j}} \sum_{\substack{\Lambda \in \p(S^C) \\ \Lambda \neq \emptyset}} \sum_{\pi \in {\rm Part}
(\Lambda )
}
\eps^{|\Lambda | - |\pi| - l + |S|}(-1)^{|\pi|} \frac{(m+ |\pi|)!}{m!}\\
&\cdot C(|S \cup \Lambda|)\frac{(y_N +1)^{m+1-l + |S| +|\Lambda|}}{(\eps^{\alpha-1} + 1)^{m+1+|\pi|}} \prod_{k \in (S^C \setminus \Lambda)} \delta_{i_k N} \prod_{B \in \pi}  \epsilon ^{\alpha -|B|}  \frac{\partial^{|B|}b(\bar{y})}{\prod_{l \in B} \partial y_{i_l}}.
\end{split}
\end{equation}}
It is possible to prove by direct computation that all the summands appearing in the second line in the right-hand side of \eqref{proof: contih final}  are vanishing as $\eps \to 0$.
%
%
By letting $\eps \to 0$ in \eqref{proof: contih final} we see that
{\small
\[
\begin{split}
\lim_{\eps \to 0} \epsilon^{l-\alpha} \widehat{\frac{\partial^l h_\epsilon}{\partial x_{i_1} \cdots \partial x_{i_l}}} (\bar{x},y) &= \sum_{S \in \p(l)} \frac{\partial^{|S|} b(\bar{y})}{\prod_{j \in S} \partial y_{i_j}} C(|S|) \,(y_N + 1)^{m+1-l + |S|} \prod_{j \notin S} \delta_{{i_j}N}\\
 &=\frac{\partial^l}{\partial y_{i_1} \cdots \partial y_{i_l}}\big( b(\bar{y})(y_N + 1)^{m+1} \big),
\end{split}
\]}
concluding the proof.
\end{proof}

Finally, we are ready to prove the following 

\begin{proposition}
\label{prop: macro limit}
Let $v_\epsilon \in V(\Omega_\epsilon)$ be such that $\norma{v_\epsilon}_{W^{m,2}(\Omega_\epsilon)} \leq M$ for all $\epsilon>0$.  Let $\widetilde{Y} = Y \times (-1,0)$ and $g(y) = b(\bar{y})(1 + y_N)^{m+1}$ for all $y \in \widetilde{Y}$. Moreover, let $\hat{v}\in L^2(W, w^{m,2}_{Per_Y}(Y\times (-\infty,0)))$ be as in Lemma~\ref{lemma: unfolding convergence poly}. Then
{\small \begin{multline*}
\int_{Q_\epsilon} D^m v_{\epsilon} : D^m(T_\epsilon \varphi) \, \diff{x} \rightarrow\\
 - \sum_{l=1}^{m-1} \binom{m}{l+1} \int_W \int_{\widetilde{Y}} \frac{y_N^{l-1}}{(l-1)!} D_y^{l+1} \bigg(\frac{\partial^{m-l-1} \hat{v}(\bar{x},y)}{\partial y_N^{m-l-1}} \bigg) : D^{l+1}_y g(y) \,\diff{y}\, \frac{\partial^{m-1}\varphi}{\partial x_N^{m-1}}(\bar{x}, 0) \diff{\bar{x}},
\end{multline*}
}for all $\varphi \in W^{m,2}(\Omega) \cap W^{m-1,2}_0(\Omega)$, as $\eps \to 0$.
\end{proposition}
\begin{proof}
We set
\begin{align*}
P_1(t) &= \{ \pi = (S_1, \dotsc, S_t) \in {\rm Part}(\{1, \dotsc, m\}) : \textup{$\exists! \ S_k$ with $|S_k| > 1$} \},\\
P_2(t) &= \{ \pi \in {\rm Part}(\{1, \dotsc, m\}) : |\pi| = t, \pi \notin P_1(t) \}.
\end{align*}
We note that in the definition of $P_1(t)$ we may assume without loss of generality that the only element $S_k$ with cardinality strictly bigger than 1 is $S_1$.
In the sequel, we always assume that a given partition $\pi$ of cardinality $t$ is represented by $\pi = \{S_1, \dotsc, S_t\}$. In the following calculations,  we use the index notation and we drop the summation symbols  $\sum_{j_1,\dots , j_{|\pi |} =1}^ N  $ and $\sum_{i_1, \cdots, i_m= 1}^N$.  With the help of \eqref{eq: Faa di Bruno mdim}  we compute
\begin{equation}
\label{proof: A1}
\begin{split}
&\int_{Q_\epsilon} D^m v_{\epsilon} : D^m(T_\epsilon \varphi) \, \diff{x} = \int_{Q_\epsilon} \frac{\partial^m v_\epsilon}{\partial x_{i_1} \cdots \partial x_{i_m}} \frac{\partial^m (\varphi \circ \Phi_\epsilon)}{\partial x_{i_1} \cdots \partial x_{i_m}} \, \diff{x}\\
&= \sum_{\substack{ \pi \in {\rm Part}(\{1, \dotsc, m\})\\ \pi= \{S_1, \dotsc, S_{|\pi|}\}}} \int_{Q_\epsilon} \frac{\partial^m v_\epsilon}{\partial x_{i_1} \cdots \partial x_{i_m}} \frac{\partial^{|\pi|} \varphi}{\prod_{k=1}^{|\pi|} \partial x_{j_k}}(\Phi_\epsilon(x))\, \prod^{|\pi|}_{k=1} \frac{\partial^{|S_k|} \Phi_\epsilon^{(j_k)}}{\prod_{l \in S_k} \partial x_{i_l}} \diff{x}\\
&= \int_{Q_\epsilon} \frac{\partial^m v_\epsilon}{\partial x_{i_1} \cdots \partial x_{i_m}}  \frac{\partial^m \varphi}{\partial x_{j_1} \cdots \partial x_{j_m}}(\Phi_\epsilon(x)) \frac{\partial \Phi_\eps^{(j_1)}}{\partial x_{i_1}} \cdots \frac{\partial \Phi_\eps^{(j_m)}}{\partial x_{i_m}}    \diff{x},\\
&+ \sum_{t=1}^{m-1} \sum_{\pi \in P_1(t)} \int_{Q_\epsilon} \frac{\partial^m v_\epsilon}{\partial x_{i_1} \cdots \partial x_{i_m}} \frac{\partial^t \varphi}{\prod_{k=1}^t\partial x_{j_k}}(\Phi_\epsilon(x))\, \prod^{t}_{k=1} \frac{\partial^{|S_k|} \Phi_\epsilon^{(j_k)}}{\prod_{l \in S_k} \partial x_{i_l}} \diff{x}\\
&+ \sum_{t=2}^{m-2} F_t(v_\eps, \varphi, \Phi_\eps),
\end{split}
\end{equation}
where $F_t(v_\eps, \varphi, \Phi_\eps)$ is defined by
\[
F_t(v_\eps, \varphi, \Phi_\eps) = \sum_{\pi \in P_2(t)} \int_{Q_\eps} \frac{\partial^m v_\eps}{\partial x_{i_1} \cdots \partial x_{i_m}} \frac{\partial^{t} \varphi}{\prod_{k=1}^t \partial x_{j_k}} \prod^t_{k=1} \frac{\partial^{|S_k|} \Phi_\eps^{(j_k)}}{\prod_{l \in S_k} \partial x_{i_l}} dx.
\]
We consider separately the three summands in the right hand side of~\eqref{proof: A1}. Let us remark for future use that
\[
\frac{\partial \Phi_\epsilon^{(k)}}{\partial x_i } =
\begin{cases}
\delta_{ki}, &\textup{if $k \neq N$},\\
\delta_{Ni} - \frac{\partial h_\epsilon}{\partial x_i}, &\textup{if $k=N$},
\end{cases} \quad\quad
\frac{\partial^l \Phi_\epsilon^{(k)}}{\partial x_{i_1} \cdots \partial x_{i_l} } =
\begin{cases}
0, &\textup{if $k \neq N$},\\
-\frac{\partial^l h_\epsilon}{\partial x_{i_1} \cdots \partial x_{i_l}}, &\textup{if $k=N$}.
\end{cases}
\]
for all $2 \leq l \leq m$. Consider now the first term in the right hand side of~\eqref{proof: A1}. We unfold it by taking into account \eqref{eq: exact int formula} in order to obtain
\[
\begin{split}
&\Bigg\lvert\epsilon \int_{\hat{W}_\eps} \int_{\widetilde{Y}} \widehat{\frac{\partial^m v_\epsilon}{\partial x_{i_1} \cdots \partial x_{i_m}}} \frac{\partial^m\varphi}{\partial x_{j_1} \cdots \partial x_{j_m}} (\hat{\Phi}_\epsilon(y))\, \widehat{\frac{\partial \Phi_\epsilon^{(j_1)}}{\partial x_{i_1}}} \cdots \widehat{\frac{\partial \Phi_\epsilon^{(j_m)}}{\partial x_{i_m}}} \, dy d\bar{x} \Bigg\rvert\\
&= \epsilon^{-2m + 1} \Bigg\lvert\int_{\hat{W}_\eps} \int_{\widetilde{Y}} \frac{\partial^m \hat{v}_\epsilon}{\partial y_{i_1} \cdots \partial y_{i_m}} \frac{\partial^m\varphi}{\partial x_{j_1} \cdots \partial x_{j_m}} (\hat{\Phi}_\epsilon(y))\, \frac{\partial \widehat{\Phi}_\epsilon^{(j_1)}}{\partial y_{i_1}} \cdots\frac{\partial \widehat{\Phi}_\epsilon^{(j_m)}}{\partial y_{i_m}} \, dy d\bar{x} \Bigg\rvert\\
&\leq C \epsilon^{-m + 1} \eps^{m-1/2} \int_{\hat{W}_\eps} \int_{\widetilde{Y}} \Bigg\lvert\eps^{-m+1/2}\frac{\partial^m \hat{v}_\epsilon}{\partial y_{i_1} \cdots \partial y_{i_m}} \frac{\partial^m\varphi}{\partial x_{j_1} \cdots \partial x_{j_m}} (\hat{\Phi}_\epsilon(y))\Bigg\rvert \, dy d\bar{x} \\
&\leq C \epsilon^{1/2} \norma*{\eps^{-m+1/2}\frac{\partial^m \hat{v}_\epsilon}{\partial y_{i_1} \cdots \partial y_{i_m}}}_{L^2(\hat{W}_\eps \times \widetilde{Y})} \norma*{\frac{\partial^m\varphi}{\partial x_{j_1} \cdots \partial x_{j_m}} (\hat{\Phi}_\epsilon(y))}_{L^2(\hat{W}_\eps \times \widetilde{Y})}\\
&\leq C \eps^{1/2} \norma*{\frac{\partial^m\varphi}{\partial x_{j_1} \cdots \partial x_{j_m}} (\hat{\Phi}_\epsilon(y))}_{L^2(\hat{W}_\eps \times \widetilde{Y})} \leq C \norma*{\frac{\partial^m\varphi}{\partial x_{j_1} \cdots \partial x_{j_m}}}_{L^2(\Phi_\eps(Q_\eps))},
\end{split}
\]
which vanishes as $\epsilon \to 0$. In the first inequality we have used the fact that $ \abs*{\frac{\partial \hat{\Phi}_\epsilon^{(k)}}{\partial y_i}} \leq C \eps,
$
for sufficiently small $\eps>0$.
Let now $1 \leq t \leq m-1$ be fixed and consider
{\small
\begin{multline}
\label{proof: A3}
\sum_{\pi \in P_1(t)} \int_{Q_\epsilon} \frac{\partial^m v_\epsilon}{\partial x_{i_1} \cdots \partial x_{i_m}} \frac{\partial^t \varphi}{\prod_{k=1}^t\partial x_{j_k}}(\Phi_\epsilon(x))\, \prod^t_{k=1} \frac{\partial^{|S_k|} \Phi_\epsilon^{(j_k)}}{\prod_{l \in S_k} \partial x_{i_l}} \diff{x}\\
= \sum_{\pi \in P_1(t)} \int_{Q_\epsilon} \frac{\partial^m v_\epsilon}{\partial x_{i_1} \cdots \partial x_{i_m}} \frac{\partial^t \varphi}{\prod_{k=1}^t\partial x_{j_k}}(\Phi_\epsilon(x))\, \frac{\partial^{m-t+1} \Phi_\eps^{(j_1)}}{\prod_{l \in S_1} \partial x_{i_l}} \frac{\partial \Phi_\eps^{(j_2)}}{\partial x_{i_{S_2}}} \cdots \frac{\partial \Phi_\eps^{(j_t)}}{\partial x_{i_{S_t}}}  \diff{x},
\end{multline}
}
where to shorten the notation we have identified $S_2, \dots, S_t$ with the only element they contain. Note that if ${j_1}\neq N$ then the integral in \eqref{proof: A3} is zero. Thus, without loss of generality we set $j_1=N$. Note that we have $\frac{\partial \Phi_\eps^{(N)}}{\partial x_{i_t}} = \delta_{N {i_t}} + \frac{\partial h_\epsilon}{\partial x_{i_t}} $ and $ \abs*{\frac{\partial h_\epsilon}{\partial x_{i_t}}} \leq C \eps^{1/2}$ as $\eps \to 0$.
In order to simplify the expressions we will not write down the higher order terms in $\eps$. Hence, by setting $j_1= N$ in \eqref{proof: A3} we deduce that the lower order terms in   \eqref{proof: A3} are given by
{\small
\begin{equation}
\label{proof: A4}
\begin{split}
&\sum_{\pi \in P_1(t)} \int_{Q_\epsilon} \frac{\partial^m v_\epsilon}{\partial x_{i_1} \cdots \partial x_{i_m}} \frac{\partial^t \varphi}{ \partial x_{N}\partial x_{j_{S_2}}\cdots \partial x_{j_{S_t}}  }(\Phi_\epsilon)\, \frac{\partial^{m-t+1} \Phi_\eps^{(N)}}{\prod_{l \in S_1} \partial x_{i_l}} \delta_{i_{S_2}j_2} \cdots \delta_{i_{S_t}j_N} \diff{x} \\
&= \sum_{\pi \in P_1(t)} \int_{Q_\epsilon} \frac{\partial^t \varphi}{\partial x_{N}\partial x_{i_{S_2}}\cdots \partial x_{i_{S_t}}  }  (\Phi_\epsilon) \, \frac{\partial^m v_\epsilon}{\prod_{l \in S_1} \partial x_{i_l} \partial x_{i_{S_2}}\cdots \partial x_{i_{S_t}} }       \, \frac{\partial^{m-t+1} \Phi_\eps^{(N)}}{\prod_{l \in S_1} \partial x_{i_l}} \,\diff{x}\\
&= \mybinom{m}{t-1} \! \int_{Q_\epsilon}\! \frac{\partial^t \varphi}{ \partial x_{N} \partial x_{i_{S_2}}\cdots \partial x_{i_{S_t}}  }   (\Phi_\epsilon) \, \frac{\partial^m v_\epsilon}{\prod_{l \in S_1} \! \partial x_{i_l} \partial x_{i_{S_2}}\cdots \partial x_{i_{S_t}} }\, \frac{\partial^{m-t+1} \Phi_\eps^{(N)}}{\prod_{l \in S_1} \!\partial x_{i_l}} \,\diff{x},
\end{split}
\end{equation}
}
where in the last equality in \eqref{proof: A4} we have used the fact that each of the summands
\[
\int_{Q_\epsilon} \frac{\partial^t \varphi}{\partial x_{N}\partial x_{i_{S_2}}\cdots \partial x_{i_{S_t}}  }(\Phi_\epsilon) \, \frac{\partial^m v_\epsilon}{\prod_{l \in S_1} \partial x_{i_l} \partial x_{i_{S_2}}\cdots \partial x_{i_{S_t}}   } \, \frac{\partial^{m-t+1} \Phi_\eps^{(N)}}{\prod_{l \in S_1} \partial x_{i_l}} \,\diff{x}
\]
equals
\[
\int_{Q_\epsilon} \frac{ \partial^t\varphi   } { \partial x_{N}\partial x_{i_{S_2}}\cdots \partial x_{i_{S_t}}  }    (\Phi_\epsilon) \, D^{m-t+1}\Bigg(\frac{\partial^{t-1} v_\eps}{ \partial x_{i_{S_2}}\cdots \partial x_{i_{S_t}}   } \Bigg) : D^{m-t+1} \Phi_\eps^{(N)} \,\diff{x}.
\]
and in particular they do not depend on the choice of $\pi$ (note that the cardinality of $P_1(t)$ equals $\binom{m}{t-1}$). By unfolding the right-hand side of \eqref{proof: A4} and using the fact that $m-t+1\geq 2$ we have that
{\footnotesize
\begin{equation}
\label{proof: A6}
\begin{split}
&\binom{m}{t-1} \eps  \int_{\hat{W}_\eps} \int_{\widetilde{Y}} \frac{\partial^t \varphi}{\partial x_N \partial x_{i_{S_2}}\cdots \partial x_{i_{S_t}}   }(\hat{\Phi}_\epsilon(y)) \, \widehat{\frac{\partial^m v_\epsilon}{\prod_{l \in S_1} \partial x_{i_l}    \partial x_{i_{S_2}}\cdots \partial x_{i_{S_t}}   } }\, \widehat{\frac{\partial^{m-t+1} \Phi_\eps^{(N)}}{\prod_{l \in S_1} \partial x_{i_l}}} dy d\bar{x}\\
&= -\binom{m}{t+1} \frac{\eps}{\eps^m} \int_{\hat{W}_\eps} \int_{\widetilde{Y}} \frac{\partial^m \hat{v}_\epsilon}{\prod_{l \in S_1} \partial y_{i_l}  \partial y_{i_{S_2}}\cdots \partial y_{i_{S_t}}} \frac{\partial^{t}\varphi}{ \partial x_N \partial x_{i_{S_2}}\cdots \partial x_{i_{S_t}}   } (\hat{\Phi}_\epsilon(y)) \widehat{\frac{\partial^{m-t+1} h_\eps}{\prod_{l \in S_1} \partial x_{i_l}}} dy d\bar{x}.
\end{split}
\end{equation}}
It is easy to see that the final expression appearing in the right-hand side of \eqref{proof: A6} can be written as
{\small
\begin{multline}
\label{proof: A5}
-\binom{m}{t+1} 
\int_{\hat{W}_\eps} \int_{\widetilde{Y}} \Bigg[\eps^{-m+1/2}\frac{\partial^m \hat{v}_\epsilon}{\prod_{l \in S_1} \partial y_{i_l}  \partial y_{i_{S_2}}\cdots \partial y_{i_{S_t}}   }  \Bigg]\\
\cdot \Bigg[\frac{1}{\eps^{m-t-1}}\frac{\partial^{t}\varphi}{\partial x_{N} \partial y_{i_{S_2}}\cdots \partial y_{i_{S_t}}   } (\hat{\Phi}_\epsilon(y))\Bigg] \Bigg[ \eps^{m-t+1 -3/2}\widehat{\frac{\partial^{m-t+1} h_\eps}{\prod_{l \in S_1} \partial x_{i_l}}}\Bigg] dy d\bar{x}.
\end{multline}}
Now
\[
\eps^{-m+1/2}\frac{\partial^m \hat{v}_\epsilon}{\prod_{l \in S_1} \partial y_{i_l}  \partial y_{i_{S_2}}\cdots \partial y_{i_{S_t}}   }      \to \frac{\partial^m \hat{v}}{\prod_{l \in S_1} \partial y_{i_l}\partial y_{i_{S_2}}\cdots \partial y_{i_{S_t}}   },
\]
weakly in $L^2(\widehat{W}_\eps \times Y \times (-1,0))$ as $\eps \to 0$, by Lemma \ref{lemma: unfolding convergence poly}, and
\[
\eps^{m-t+1 -3/2}\widehat{\frac{\partial^{m-t+1} h_\eps}{\prod_{l \in S_1} \partial x_{i_l}}} \to \frac{\partial^{m-t+1} (b(\bar{y})(1+y_N)^{m+1})}{\prod_{l \in S_1} \partial y_{i_l}},
\]
in $L^\infty(\widehat{W}_\eps \times Y \times (-1,0))$ as $\eps \to 0$, by Lemma \ref{lemma: convergence b poly}. Moreover,  by Lemma \ref{lemma: auxiliary} in the Appendix it follows  that
\[
\frac{1}{\eps^{m-t-1}}\frac{\partial^{t}\varphi}{\partial x_N^{t}} (\hat{\Phi}_\epsilon(y))\to \frac{y_N^{m-t-1}}{(m-t-1)!} \frac{\partial^{m-1}\varphi}{\partial x_N^{m-1}}(\bar{x},0),
\]
and
\[
\frac{1}{\eps^{m-t-1}}\frac{\partial^{t}\varphi}{\partial x_N \partial x_{i_{S_2}}\cdots \partial x_{i_{S_t}}   }   (\hat{\Phi}_\epsilon(y)) \to 0,
\]
strongly in $L^2(W \times Y \times (-1,0))$ as $\eps \to 0$, if at least one  of the indexes $ i_{S_2}, \dots , i_{S_N} $ is not equal to $N$. Hence \eqref{proof: A5} tends to
{\small \begin{multline*}
-\binom{m}{t+1} \int_{W} \int_{Y \times (-1,0)} \frac{y_N^{m-t-1}}{(m-t-1)!} D_y^{m-t+1}\Bigg(\frac{\partial^{t-1}\hat{v}}{\partial y_N^{t-1}} \Bigg) : D_y^{m-t+1}\Big( b(\bar{y})(1+y_N)^{m+1}\Big) dy \frac{\partial^{m-1}\varphi}{\partial x_N^{m-1}}(\bar{x},0) d\bar{x}.
\end{multline*}}
By setting $m-t = l$ we recover the limiting expression in the statement. Then, in order to conclude the proof it is sufficient to prove that the integrals in $F_t(v_\eps, \varphi, \Phi_\eps)$ vanish as $\eps \to 0$.  We will show this by comparing each integral appearing in the definition of $F_t(v_\eps, \varphi, \Phi_\eps)$ with the corresponding integral of the form \eqref{proof: A3}, which is convergent as $\eps \to 0$, hence it is uniformly bounded in $\eps$. Note that  by Lemma \ref{lemma: convergence b poly}
\[
\frac{\partial^{m-t+1} \hat{\Phi}_\eps^{(j_1)}}{\prod_{l \in S_1} \partial y_{i_l}} \frac{\partial \hat{\Phi}_\eps^{(j_2)}}{\partial y_{i_{S_2}}}\cdots \frac{\partial \hat{\Phi}_\eps^{(j_{t})}}{\partial y_{i_{S_t}}} = O(\eps^{3/2 + t -1}) = O(\eps^{1/2 + t}),
\]
for all $\pi \in P_1(t)$,  whereas if we consider $\pi'  =(S_1',\dots , S_t') \in P_2(t)$ with $|S'_1| = m - t < m-t +1$ there must exists $S'_k$, $k > 1$ with $|S'_k| = 2$. Let us assume that $k=2$. Then we have
\[
\frac{\partial^{m-t} \hat{\Phi}_\eps^{(j_1)}}{\prod_{l \in S'_1} \partial y_{i_l}} \frac{\partial^2 \hat{\Phi}^{(j_2)}_\eps}{\prod_{l \in S'_2} \partial y_{i_l}}  \frac{\partial \hat{\Phi}_\eps^{(j_3)}}{\partial y_{i_{S'_3}}}\cdots \frac{\partial \hat{\Phi}_\eps^{(j_{t})}}{\partial y_{i_{S'_t}}} = O(\eps^{3/2 + t } \eps^{3/2-2} ) = O(\eps^{1 + t}),
\]
and since $\eps^{1+ t} = o\big(\eps^{1/2 + t}\big)$ as $\eps \to 0$ and the integral \eqref{proof: A3} is bounded, we deduce that the integral in $F_t(v_\eps, \varphi, \Phi_\eps)$ involving
\[
\frac{\partial^m v_\eps}{\partial x_{i_1} \cdots \partial x_{i_m}} \frac{\partial^{t} \varphi}{\prod_{k=1}^t \partial x_{j_k}} \frac{\partial^{m-t} \hat{\Phi}_\eps^{(j_1)}}{\prod_{l \in S'_1} \partial y_{i_l}} \frac{\partial^2 \hat{\Phi}^{(j_2)}_\eps}{\prod_{l \in S'_2} \partial y_{i_l}}  \frac{\partial \hat{\Phi}_\eps^{(j_3)}}{\partial y_{i_{S'_3}}}\cdots \frac{\partial \hat{\Phi}_\eps^{(j_{t})}}{\partial y_{i_{S'_t}}},
\]
for all $\pi' \in P_2(t)$ defined above, vanishes as $\eps \to 0$. By arguing in a similar way for all the terms in $F_t(v_\eps, \varphi, \Phi_\eps)$ we deduce the validity of the statement.
\end{proof}

We summarise the previous discussion in the following
\begin{theorem}
\label{thm: macroscopic limit poly}
Let $f_\epsilon \in L^2(\Omega_\epsilon)$, $f\in L^2(\Omega)$ be such that $f_\epsilon \rightharpoonup f$ in $L^2(\Omega)$. Let $g(y) = b(\bar{y})(1+y_N)^{m+1}$ for all $y \in Y \times (-1,0)$. Moreover, let us assume that $v_\epsilon\in W^{m,2}(\Omega_\epsilon) \cap W^{m-1,2}_0(\Omega_\epsilon)$ is the solution to $H_{\Omega_\epsilon,S} v_\epsilon = f_\epsilon$ for all $\eps > 0$. Then there exist $v\in W^{m,2}(\Omega) \cap W^{m-1,2}_0(\Omega)$ and a function $\hat{v}$ in the space $L^2(W,w^{m,2}_{Per_Y}(Y\times(-\infty,0)))$ such that, possibly passing to a subsequence,  $v_\epsilon\rightharpoonup v$ in $W^{m,2}(\Omega)$, $v_\epsilon \to v$ in $L^2(\R^N)$, and statements $(a)$ and $(b)$ in Lemma~\ref{lemma: unfolding convergence poly} hold. Moreover, the following integral equality holds
{\small\begin{multline} \label{conclusion theorem macro}
- \sum_{l=1}^{m-1} \binom{m}{l+1} \int_W \int_{Y \times (-1,0)} \Bigg[\frac{y_N^{l-1}}{(l-1)!} D_y^{l+1} \Bigg(\frac{\partial^{m-l-1} \hat{v}(\bar{x},y)}{\partial y_N^{m-l-1}} \Bigg): D^{l+1}_y g(y)\Bigg] \,\diff{y}\, \frac{\partial^{m-1}\varphi}{\partial x_N^{m-1}}(\bar{x}, 0)  \diff{\bar{x}} \\+\int_{\Omega} D^m v : D^m\varphi + u\varphi \,\diff{x} = \int_{\Omega} f \varphi\, \diff{x}.
\end{multline}}
for all $\varphi \in W^{m,2}(\Omega) \cap W^{m-1,2}_0(\Omega)$.
\end{theorem}
\vspace{6pt}
\noindent\textit{Notation.} We will use the following notation:
\[ 
q_Y(f, g) :=  \sum_{l=1}^{m-1} \binom{m}{l+1} \int_{Y \times (-1,0)} \Bigg[\frac{y_N^{l-1}}{(l-1)!} D_y^{l+1} \Bigg(\frac{\partial^{m-l-1} f(\bar{x},y)}{\partial y_N^{m-l-1}} \Bigg) : D^{l+1}_y g(y)\Bigg] \,\diff{y}
\]
for all $f\in L^2(W,w^{m,2}_{Per_Y}(Y\times(-\infty,0)))$, $g \in C^m_{Per_Y}(Y \times (-1,0))$. We refer to 
\begin{equation} \label{strange term poly}
-\int_W q_Y(\hat{v}, g) \frac{\partial^{m-1}\varphi}{\partial x_N^{m-1}}(\bar{x}, 0)  \diff{\bar{x}}
\end{equation}
 as the   \textit{strange term} appearing in the homogenization.


\subsection{Critical case - Microscopic problem.}
The aim of this section is to characterize the strange term \eqref{strange term poly} as the energy of a suitable polyharmonic function  and in particular  to conclude that it is different from zero. We will use periodically oscillating test functions matching the intrinsic $\eps$-scaling of the problem.\\
Let then $\psi \in C^{\infty}(\overline{W} \times \overline{Y} \times ]-\infty, 0])$ be such that $\supp{\psi} \subset C \times \overline{Y} \times [d,0]$ for some compact set $C\subset W$ and for some $d\in (-\infty, 0)$. Moreover, assume that $\psi(\bar{x},\bar{y},0) = D^l\psi(\bar{x},\bar{y},0) = 0$ for all $(\bar{x},\bar{y})\in W \times Y$, for all $1 \leq l \leq m-2$ . Let also $\psi$ be $Y$-periodic in the variable $\bar{y}$. We set
\[
\psi_\epsilon(x) = \epsilon^{m-\frac{1}{2}} \psi \Big(\bar{x},\frac{\bar{x}}{\epsilon}, \frac{x_N}{\epsilon}\Big),
\]
for all $\epsilon >0$, $x \in W \times ]-\infty, 0]$. Then $T_\epsilon \psi_\epsilon \in V(\Omega_\epsilon)$ for all sufficiently small $\epsilon$, hence we can use it as a test function in the weak formulation of the problem in $\Omega_\epsilon$, getting
\[
\int_{\Omega_\epsilon} D^m v_\epsilon : D^m T_\epsilon \psi_\epsilon\,\diff{x} + \int_{\Omega_\epsilon} v_\epsilon T_\epsilon \psi_\epsilon \,\diff{x} = \int_{\Omega_\epsilon} f_\epsilon T_\epsilon \psi_\epsilon \,\diff{x}.
\]
It is not difficult to prove that
\begin{equation}
\label{eq: micro poly vanish}
\int_{\Omega_\eps} v_\eps T_\eps \psi_\eps \, \diff{x} \to 0, \quad \quad \int_{\Omega_\eps} f_\eps T_\epsilon \psi_\epsilon \,\diff{x} \to 0
\end{equation}
as $\eps \to 0$. By arguing as in \cite[\S 8.4]{ArrLamb}, it is also possible to prove that
\begin{equation}
\label{eq: micro poly vanish 2}
\int_{\Omega_\eps \setminus \Omega} D^m v_\eps : D^m T_\epsilon \psi_\epsilon\,\diff{x} \to 0,
\end{equation}
as $\eps \to 0$. Moreover, a suitable modification of \cite[Lemma 8.47]{ArrLamb} yields
\begin{equation}
\label{eq: micro poly conv}
\int_\Omega D^m v_\epsilon : D^m T_\epsilon \psi_\epsilon\,\diff{x} \to \int_{W\times Y\times (-\infty,0)} D_y^m \hat{v}(\bar{x},y) : D_y^m \psi(\bar{x},y)\,\diff{\bar{x}} \diff{y}.
\end{equation}

\begin{theorem}
\label{thm: conditions on v poly}
Let $\hat{v} \in L^2(W,w^{m,2}_{Per_Y}(Y\times(\infty,0)))$ be the function from Theorem~\ref{thm: macroscopic limit poly}. Then
\begin{equation}
\label{eq: variational v hat poly}
\int_{W\times Y \times (-\infty,0)} D^m_y\hat{v}(\bar{x},y) : D^m_y\psi (\bar{x},y)\,\, \diff{\bar{x}}\, \diff{y} = 0,
\end{equation}
for all $\psi \in L^2(W,w^{m,2}_{Per_Y}(Y\times(\infty,0)))$ such that $\psi(\bar{x},\bar{y},0) = D^l_y\psi(\bar{x},\bar{y},0) = 0$ for all $(\bar{x},\bar{y})\in W \times Y$, for all $1 \leq l \leq m-2$. Moreover, for any $j=1,\dotsc, N-1$, we have
\begin{equation}
\label{eq: Casado1 poly}
\frac{\partial^{m-1}\hat{v}}{\partial y_j \partial y^{m-2}_N}(\bar{x}, \bar{y},0) = - \frac{\partial b}{\partial y_j}(\bar{y}) \frac{\partial^{m-1} v}{\partial x^{m-1}_N}(\bar{x}, 0), \quad\quad \textup{on $W \times Y$},
\end{equation}
and
\begin{equation}
\label{eq: Casado2 poly}
\frac{\partial^{m-1}\hat{v}}{\partial y_{i_1} \cdots \partial y_{i_{m-1}}}(\bar{x}, \bar{y},0) = 0, \quad\quad \textup{on $W \times Y$},
\end{equation}
for all $i_1,\dotsc, i_{m-1} = 1,\dots, N-1$.
\end{theorem}
\begin{proof}
The first part of the statement follows from \eqref{eq: micro poly vanish}, \eqref{eq: micro poly vanish 2} and \eqref{eq: micro poly conv} by arguing as in \cite[Theorem 8.53]{ArrLamb}. In order to prove formulas~\eqref{eq: Casado1 poly} and~\eqref{eq: Casado2 poly} we note that, since $D^{m-2} v_\epsilon(\bar{x},g_\epsilon(\bar{x}))= 0$ for all $\bar{x}\in W$, we have
\[
\frac{\partial^{m-2} v_\epsilon}{\partial x_{i_1}\cdots \partial x_{i_{m-2}}} (\bar{x},g_\epsilon(\bar{x})) = 0, \quad \textup{ for all $i_1, \dotsc, i_{m-2} = 1,\dots,N$, $\bar{x}\in W$}.
\]
Differentiating with respect to $x_j$, $j \in \{1,\dots, N-1 \}$ yields
\[
\frac{\partial^{m-1} v_\epsilon}{\partial x_{i_1}\cdots \partial x_{i_{m-2}}\partial x_j } (\bar{x},g_\epsilon(\bar{x})) + \frac{\partial^{m-1} v_\epsilon}{\partial x_{i_1}\cdots \partial x_{i_{m-2}} \partial x_N} (\bar{x},g_\epsilon(\bar{x})) \frac{\partial g_\epsilon(\bar{x})}{\partial x_j} = 0,
\]
for all $\bar{x}\in W$. Hence, by setting
\[
V_\epsilon^{j} = \Big(0,\dots,0,-\frac{\partial^{m-1} v_\epsilon}{ \partial x_N\partial x_{i_1}\cdots \partial x_{i_{m-2}}}, 0,\dots,0, \frac{\partial^{m-1} v_\epsilon}{\partial x_j \partial x_{i_1}\cdots \partial x_{i_{m-2}}}\Big),
\]
for all $i_1,\dotsc, i_{m-2} = 1,\dotsc, N$, $j=1,\dotsc,N-1$, where the only non-zero entries are the $j$-th and the $N$-th, we obtain that
$
V_\epsilon^{j} \cdot n_\epsilon = 0$,  on $\Gamma_\epsilon$,
where $n_\epsilon$ is the outer normal to $\Gamma_\epsilon \equiv \{(\bar{x}, g_\eps(\bar{x})) : \bar{x} \in W\}$. By using Lemma~\ref{lemma: unfolding convergence poly}
\[\frac{
   \frac{\widehat{\partial^{m-1} v_\epsilon}}{\partial x_{i_1} \cdots \partial x_{i_{m-2}} \partial x_j}
   -\int_Y  \frac{\widehat{\partial^{m-1} v_\epsilon}}{\partial x_{i_1} \cdots \partial x_{i_{m-2}} \partial x_j} (\bar x, \bar y, 0)d\bar y
     }
 { \sqrt{\epsilon}   }
  \overset{\epsilon \to 0}{\rightharpoonup} \frac{\partial^{m-1} \hat{v}}{\partial y_{i_1}\cdots \partial y_{i_{m-2}}\partial y_j },
\]
in $L^2(W \times Y \times ]d, 0[)$ for any $d<0$. This combined with \cite[Lemma 4.3]{CasDiaz} (see also \cite[Lemma~8.56]{ArrLamb}) yields
\[
\frac{\partial^{m-1}\hat{v}}{\partial y_{i_1}\cdots \partial y_{i_{m-2}}\partial y_j}(\bar{x}, \bar{y},0) = - \frac{\partial b}{\partial y_j}(\bar{y}) \frac{\partial^{m-1} v}{\partial x_N \partial x_{i_1} \cdots \partial x_{i_{m-2}}}(\bar{x}, 0),  
\]
for all $(\bar{x}, \bar{y}) \in W \times Y$, $i_1,\dotsc,i_{m-2}=1,\dots,N$, $j=1,\dots,N-1$. We deduce that since $v \in W^{m,2}(\Omega) \cap W^{m-1,2}_0(\Omega)$, then $D^{m-2}\, v(\bar{x},0) = 0$ for all $x \in W$. This implies that all the derivatives
$
\frac{\partial^{m-1} v}{\partial x_N \partial x_{i_1} \cdots \partial x_{i_{m-2}}}(\bar{x}, 0),
$
where one of the indexes $i_k$ is different from $N$, are zero. This concludes the proof.
\end{proof}

Now we have the following
\begin{lemma}
\label{lemma: V poly}
There exists $V\in w^{m,2}_{Per_Y}(Y \times (-\infty,0))$ satisfying the equation
\begin{equation}
\label{eq: variational V poly}
\int_{Y \times (-\infty,0)} D^m V : D^m \psi \,\,\diff{y} = 0,
\end{equation}
for all $\psi \in w^{m,2}_{Per_Y}(Y \times (-\infty,0))$ such that $D^l \psi(\bar{y},0) = 0$ on $Y$, for all\\ $0 \leq l \leq m-2$, and the boundary conditions
\[
\begin{cases}
\frac{\partial^{l} V}{\partial y_N^{l}}(\bar{y},0) = 0,\ \ {\rm for\ all\ } l=0,\dots , m-3, &\textup{on $Y$},\\
\frac{\partial^{m-2} V}{\partial y_N^{m-2}}(\bar{y},0) = b(\bar{y}), &\textup{on $Y$}.
\end{cases}
\]
The function $V$ is unique up to the sum of a monomial  in $y_N$ of degree $m-1$ of the type $ay_N^{m-1}$ with $a\in \R$. Moreover $V \in W^{2m,2}_{Per_Y}(Y \times (d,0))$ for any $d < 0$ and
it satisfies the equation
\[
(-\Delta)^m V = 0, \quad\quad\textup{in $Y \times (d,0) $},
\]
subject to the boundary conditions
\[
\begin{cases}
\frac{\partial^l V}{\partial n^l } (\bar{y},0) = 0, &\textup{on $Y$, for all $0 \leq l \leq m-3$,}\\
\frac{\partial^{m-2} V}{\partial y_N^{m-2}}(\bar{y},0) = b(\bar{y}), &\textup{on $Y$,}\\
\frac{\partial^m V}{\partial y_N^m} (\bar{y},0) = 0,  &\textup{on $Y $}.
\end{cases}
\]
\end{lemma}
\begin{proof}
Similar to the proof of \cite[Lemma 8.60]{ArrLamb}. We just note that in order to deduce the classical formulation of problem \eqref{eq: variational V poly} it is sufficient to choose test functions $\psi$ as in the statement with bounded support in the $y_N$ direction. By using the Polyharmonic Green Formula  \eqref{polyharmonic Green formula} we then deduce that
{\small
\[
\int_{Y \times (-\infty,0)} D^m V : D^m \psi \,\diff{y} = (-1)^m \int_{Y \times (-\infty,0)} \Delta^m V \psi \,\diff{y} + \int_Y \frac{\partial^m V}{\partial y_N^m} \frac{\partial^{m-1} \psi}{\partial y_N^{m-1}}\, d\bar{y}.
\]}
By the arbitrariness of $\psi$ it is then easy to conclude the proof.
\end{proof}

\begin{theorem}
\label{thm: strange term poly}
Let $V$ be as in Lemma \ref{lemma: V poly} and $g(y) = b(\bar{y})(1 + y_N)^{m+1}$, for all $y \in Y \times (-1, 0)$. Then
\begin{equation}
\label{proof: strange term poly}
q_Y(V, g) = \int_{Y\times (-\infty,0)} |D^m V|^2\,\diff{y}.
%
\end{equation}
Furthermore
\begin{equation}
\label{eq: D^m v boundary equality}
\int_{Y\times (-\infty,0)} |D^m V|^2\,\diff{y} = - \int_{Y} \Bigg(\frac{\partial^{m-1} (\Delta V)}{\partial x_N^{m-1}} + (m-1) \Delta_{N-1} \bigg( \frac{\partial^{m-1} V}{\partial x_N^{m-1}} \bigg)   \Bigg)  b(\bar{y}) \,\diff{\bar{y}}.
\end{equation}
\end{theorem}
\begin{proof}
Let $\phi$ be the real-valued function defined on $Y \times ]-\infty,0]$ by
$\phi(y) =\frac{y^{m-2}_N}{(m-2)!} g(y)$  if $-1\leq y_N \leq 0$ and $\phi(y) =0$ if $ y_N < -1$.
Then $\phi \in W^{m,2}(Y \times (-\infty,0))$, $\frac{\partial^l \phi}{\partial y^l_N}(\bar{y},0) = 0$ for all $0 \leq l \leq m-3$, and
\begin{equation}
\label{proof: prop phi = b}
\frac{\partial^{m-2} \phi}{\partial y^{m-2}_N }(\bar{y},0) = b(\bar{y}), \quad\quad \textup{for all $y \in Y$.}
\end{equation}
Now note that the function $\psi = V - \phi$ is a suitable test-function in equation \eqref{eq: variational V poly}; by plugging it in \eqref{eq: variational V poly} we deduce that
$
\int_{Y\times (-\infty,0)} |D^m V|^2\,\diff{y} = \int_{Y \times(-1,0)} D^m V : D^m\phi\, \diff{y}.
$
By the Leibnitz rule 
we have that
\begin{multline}
\label{proof: Dm V : Dm phi}
\int_{Y \times(-1,0)} D^m V : D^m\phi\, \diff{y} \\
= \int_{Y \times (-1,0)} \frac{\partial ^m V}{\partial x_{j_1}\cdots \partial x_{j_m}}  \sum_{S \in \p(m)}\frac{1}{(m-2)!} \frac{\partial^{|S|}    y_N^{m-2}    }{\prod_{j \in S} \partial x_{i_j}} \frac{\partial^{(n-|S|)}g}{\prod_{j \notin S} \partial x_{i_j}}
  \, dy.
\end{multline}
Using the obvious fact that
\[
 \frac{\partial^{m-k} y_N^{m-2} }{\partial x_{i_1} \cdots \partial x_{i_{m-k}}}=
\begin{cases}
0, &\textup{if $k=0,1$;}\\
y_N^{k-2} \delta_{i_1 N} \cdots \delta_{i_{m-k} N}, &\textup{for $k\geq 2$.}
\end{cases}
\]
we can rewrite the  right-hand side of \eqref{proof: Dm V : Dm phi} as follows
\[
\begin{split}
&\sum_{k=2}^m \binom{m}{k} \int_{Y \times (-1,0)} D^k \bigg(\frac{\partial^{m-k} V(y)}{\partial y_N^{m-k}}\bigg) : \bigg( \frac{y_N^{k-2}}{(k-2)!} D^k g(y) \bigg)\, dy\\
&= \sum_{k=2}^m \binom{m}{k} \int_{Y \times (-1,0)} \frac{y_N^{k-2}}{(k-2)!} D^k \bigg(\frac{\partial^{m-k} V(y)}{\partial y_N^{m-k}}\bigg) : D^k g(y)\, dy,
\end{split}
\]
which coincides with the left-hand side of \eqref{proof: strange term poly} up to the change of summation index defined by $k = l+1$.
Finally, \eqref{eq: D^m v boundary equality} follows by applying the polyharmonic Green formula \eqref{polyharmonic Green formula} on $\int_{Y\times (-1,0)} D^m V : D^m \phi\,\diff{y}$. Indeed, we note that the boundary integrals on $\partial Y \times (-1,0)$ are zero, due to the periodicity of $V$ and $b$. Moreover the boundary integral on $\partial Y \times \{-1\}$ is zero since $\phi$ vanishes there together with all its derivatives. Then, the only non-trivial boundary integral is supported on $Y \times \{0\}$. More precisely, we have
{\small \begin{equation}
\int_{Y\times (-1,0)} D^m V : D^m \phi\,\diff{y} = (-1)^m \int_{Y \times (-1,0)} \Delta^m V  \phi \, \diff{y} + \sum_{t=0}^{m-1} \int_{Y} B_t(V)(\bar{y},0) \frac{\partial^t\phi(\bar{y},0)}{\partial y_N^t} \,\diff{\bar{y}},
\end{equation}}
and by recalling that $\Delta^m V = 0$ in $Y \times (-1,0)$, $\frac{\partial^m V}{\partial y_N^m} = 0$ on $Y \times \{0\}$, $\frac{\partial^l \phi}{\partial y^l_N}= 0$ on $Y \times \{0\}$, for all $0 \leq l \leq m-3$ and by \eqref{proof: prop phi = b}, we deduce that
$$
\int_{Y\times (-1,0)} D^m V : D^m \phi\,\diff{y} = \int_{Y} B_{m-2}(V)(\bar{y},0) b(\bar{y})\,\diff{\bar{y}}
$$
and by formula \eqref{polyharmonic Green formula}
$
B_{m-2}(V)(\bar{y},0) = - \sum_{l=m-2}^{m-1} \binom{l}{m-2} \Delta_{N-1}^{l-m +2} (\frac{\partial^{m-1}}{\partial y_N^{m-1}}(\Delta^{m-l-1}V) ),
$
from which we deduce \eqref{eq: D^m v boundary equality}.
\end{proof}

\begin{theorem}
\label{thm: final}
Let $m \in \N$, $m \geq 2$. Let $V$ be as in Lemma \ref{lemma: V poly}. Let $v$, $\hat{v}$ be the functions defined in Theorem \ref{thm: macroscopic limit poly}. Let also $g(y) = b(\bar{y})(1 + y_N)^{m+1}$ for all $y\in Y \times (-1,0)$. Then
\[
\hat{v}(\bar{x},y) = - V(y) \frac{\partial^{m-1} v}{\partial x_N^{m-1}}(\bar{x},0) + {a(x)y^{m-1}},
\]
for some  $a(\bar{x}) \in L^2(W)$. Moreover, the strange term \eqref{strange term poly} is given by
{\small
\[
\begin{split}
&- \int_W q_Y(\hat{v}, g) \frac{\partial^{m-1}\varphi}{\partial x_N^{m-1}}(\bar{x}, 0) \diff{\bar{x}} = \int_{Y \times (-\infty,0)} |D^m V|^2 dy \int_W  \frac{\partial^{m-1}v}{\partial x_N^{m-1}}(\bar{x}, 0)\frac{\partial^{m-1}\varphi}{\partial x_N^{m-1}}(\bar{x}, 0) \diff{\bar{x}}\\
&= - \int_{Y} \Bigg(\frac{\partial^{m-1} (\Delta V)}{\partial x_N^{m-1}} + (m-1) \Delta_{N-1} \bigg( \frac{\partial^{m-1} V}{\partial x_N^{m-1}} \bigg)  \Bigg)  b(\bar{y}) \diff{\bar{y}} \int_W  \frac{\partial^{m-1}v}{\partial x_N^{m-1}}(\bar{x}, 0)\frac{\partial^{m-1}\varphi}{\partial x_N^{m-1}}(\bar{x}, 0) \diff{\bar{x}}.
\end{split}
\]}
\end{theorem}
\begin{proof}
The proof follows by  Lemma \ref{lemma: V poly}  and Theorems \ref{thm: conditions on v poly},  \ref{thm: strange term poly} and by observing that  $- V(y) \frac{\partial^{m-1} v}{\partial x_N^{m-1}}(\bar{x},0)$ satisfies problem \eqref{eq: variational v hat poly} with the boundary conditions \eqref{eq: Casado1 poly}.
\end{proof}

We are now ready to conclude the proof of (iii) of Theorem \ref{thm: poly strong}.

\begin{proof}[Proof of Theorem \ref{thm: poly strong}(iii)]
Define $g(y) = b(\bar{y})(1+y_N)^{m+1}$ for all $y = (\bar{y}, y_N)$ in $Y \times (-1,0)$. The function $v$ in Theorem \ref{thm: macroscopic limit poly} satisfies{\small
\begin{equation}
\label{eq: poly strong limit problem omega}
\int_W \, q_Y(V,g) \frac{\partial^{m-1} v}{\partial x_N^{m-1}}(\bar{x},0)\frac{\partial^{m-1}\varphi}{\partial x_N^{m-1}}(\bar{x}, 0)  \diff{\bar{x}}
+\int_{\Omega} D^m v : D^m\varphi + u\varphi \,\diff{x} = \int_{\Omega} f \varphi\, \diff{x}.
\end{equation}
}for all $\varphi \in W^{m,2}(\Omega) \cap W^{m-1,2}_0(\Omega)$. By Theorem \ref{thm: final} we can rewrite the first integral on the left-hand side of \eqref{eq: poly strong limit problem omega} as
\[
\int_{Y \times (-\infty,0)} |D^m V|^2 dy \int_W  \frac{\partial^{m-1}v}{\partial x_N^{m-1}}(\bar{x}, 0)\frac{\partial^{m-1}\varphi}{\partial x_N^{m-1}}(\bar{x}, 0) \,\diff{\bar{x}}
\]
and by the Green Formula  \eqref{polyharmonic Green strong BC} for all $\varphi \in W^{m,2}(\Omega) \cap W^{m-1,2}_0(\Omega)$
\begin{equation}
\label{proof: poly strong trih green}
\int_{\Omega} D^m v : D^m\varphi \, dx = (-1)^m \int_{\Omega} \Delta^m v \varphi + \int_{\partial \Omega} \frac{\partial^m v}{\partial n^m} \frac{\partial^{m-1} \varphi}{\partial n^{m-1}}\, dS.
\end{equation}
Hence, in the weak formulation of the limiting problem  we find the following boundary integral
\begin{equation}
\label{proof: poly strong last B I}
\int_W \Bigg(\frac{\partial^m v}{\partial x_N^m}(\bar{x},0) + \Bigg(\int_{Y \times (-\infty,0)}|D^mV|^2 \, dy\Bigg) \frac{\partial^{m-1} v}{\partial x^{m-1}_N}(\bar{x},0)\Bigg)\frac{\partial^{m-1}\varphi}{\partial x_N^{m-1}}(\bar{x}, 0)\, \diff{\bar{x}},
\end{equation}
for all $\varphi \in W^{m,2}(\Omega) \cap W^{m-1,2}_0(\Omega)$. By \eqref{eq: poly strong limit problem omega}, \eqref{proof: poly strong trih green}, \eqref{proof: poly strong last B I} and the arbitrariness of $\varphi$ we deduce the statement of Theorem \ref{thm: poly strong}, part (iii).
\end{proof}

\section{Appendix}

In this section we prove the following technical result used in the proof of Proposition \ref{prop: macro limit}.

\begin{lemma}
\label{lemma: auxiliary}
Let $l,m\in \N$, $m\geq 2$, $1\le l\le m-1$,  $i_1, \dots , i_{m-l-1}\in \{1,\dots ,N \}$. Then  for all $\varphi \in  W^{m,2}(\Omega) \cap W^{m-1,2}_0(\Omega) $ we have
\[
\frac{1}{\eps^{l-1}}\frac{\partial^{m-l}\varphi}{\partial x_N^{m-l}} (\hat{\Phi}_\epsilon(y)) \to \frac{y_N^{l-1}}{(l-1)!} \frac{\partial^{m-1}\varphi}{\partial x_N^{m-1}}(\bar{x},0),
\]
in $L^2(W \times Y \times (-1,0)$ as $\eps \to 0$ and if at least one of the indexes $i_1, \dots , i_{m-l-1}$ does not coincide with $N$ we also have
\[
\frac{1}{\eps^{l-1}}\frac{\partial^{m-l}\varphi}{\partial x_N\partial x_{i_1} \cdots \partial x_{i_{m-l-1}} }(\hat{\Phi}_\epsilon(y)) \to 0
\]
in $L^2(W \times Y \times (-1,0)$ as $\eps \to 0$.
\end{lemma}
\begin{proof}
Note that for $l=1$ the claim follows by Lemma \ref{lemma: unfolding convergence poly}. Then assume $l > 1$. Fix $\varphi \in W^{m,2}(\Omega) \cap W^{m-1,2}_0(\Omega) \cap C^\infty(\Omega)$. Then
{\small
\begin{equation}
\label{proof: Taylor 1}
\begin{split}
&\int_{\widehat{W}_\eps \times Y \times (-1,0)} \Bigg \lvert \frac{1}{\eps^{l-1}}\frac{\partial^{m-l}\varphi}{\partial x_N^{m-l}} (\hat{\Phi}_\epsilon(y)) - \frac{y_N^{l-1}}{(l-1)!} \frac{\partial^{m-1}\varphi}{\partial x_N^{m-1}}(\bar{x},0) \Bigg \rvert^2 d\bar{x} dy \\
&= \int_{-1}^0 \sum_{k \in I_{W,\eps}} \int_{C^k_\eps} \int_Y  \Bigg \lvert \frac{1}{\eps^{l-1}}\frac{\partial^{m-l}\varphi}{\partial x_N^{m-l}} \Big(\eps \big[\frac{\bar{x}}{\eps}\big] + \eps \bar{y}, \eps y_N - h_\eps \Big(\eps \big[\frac{\bar{x}}{\eps}\big] + \eps \bar{y}, \eps y_N\Big)\Big)\\
&\qquad \qquad \qquad \qquad \qquad \qquad \qquad \qquad\qquad  - \frac{y_N^{l-1}}{(l-1)!} \frac{\partial^{m-1}\varphi}{\partial x_N^{m-1}}(\bar{x},0) \Bigg \rvert^2 d\bar{y} d\bar{x} dy_N\\
&= \int_{-1}^0 \sum_{k \in I_{W,\eps}} \int_{C^k_\eps} \int_{C^k_\eps}  \Bigg \lvert \frac{1}{\eps^{l-1}}\frac{\partial^{m-l}\varphi}{\partial x_N^{m-l}} \Big(\bar{z}, \eps y_N - h_\eps \Big(\bar{z}, \eps y_N\Big)\Big)\\
& \qquad \qquad\qquad\qquad\qquad\qquad\qquad\qquad  - \frac{y_N^{l-1}}{(l-1)!} \frac{\partial^{m-1}\varphi}{\partial x_N^{m-1}}(\bar{x},0) \Bigg \rvert^2 d\bar{x} \frac{d\bar{z}}{\eps^{N-1}} dy_N.\\
\end{split}
\end{equation}
}Now, let $\bar{z} \in C^{k}_\eps$ be fixed. By expanding $\varphi$ in Taylor's series with remainder in Lagrange form we deduce that
\[
\begin{split}
\frac{\partial^{m-l}\varphi}{\partial x_N^{m-l}} \big(\bar{z}, \eps y_N - h_\eps (\bar{z}, \eps y_N)\big) = \frac{\partial^{m-1}\varphi}{\partial x_N^{m-1}} (\bar{z}, \xi ) \frac{(\eps y_N - h_\eps(\bar{z}, \eps y_N))^{l-1}}{(l-1)!},
\end{split}
\]
for some $\xi \in (0,  \eps y_N - h_\eps (\bar{z}, \eps y_N))$. We then deduce that the term appearing in the right-hand side of \eqref{proof: Taylor 1} can be rewritten as
{\small
\begin{equation}
\label{proof: convergence weighted varphi}
\int_{-1}^0 \sum_{k \in I_{W,\eps}} \int_{C^k_\eps} \int_{C^k_\eps}  \Bigg \lvert \frac{1}{\eps^{l-1}}\frac{\partial^{m-1}\varphi}{\partial x_N^{m-1}} (\bar{z},\xi) \frac{(\eps y_N - h_\eps(\bar{z}, \eps y_N))^{l-1}}{(l-1)!} \frac{y_N^{l-1}}{(l-1)!} \frac{\partial^{m-1}\varphi}{\partial x_N^{m-1}}(\bar{x},0) \Bigg \rvert^2 d\bar{x} \frac{d\bar{z}}{\eps^{N-1}} dy_N.
\end{equation}
}
We then estimate \eqref{proof: convergence weighted varphi} from above. Note that
{\small
\begin{equation}
\label{proof: convergence weighted varphi 2}
\begin{split}
&\int_{-1}^0 \sum_{k \in I_{W,\eps}} \int_{C^k_\eps} \int_{C^k_\eps}  \Bigg \lvert \frac{1}{\eps^{l-1}}\frac{\partial^{m-1}\varphi}{\partial x_N^{m-1}} (\bar{z},\xi) \frac{(\eps y_N - h_\eps(\bar{z}, \eps y_N))^{l-1}}{(l-1)!} \frac{y_N^{l-1}}{(l-1)!} \frac{\partial^{m-1}\varphi}{\partial x_N^{m-1}}(\bar{x},0) \Bigg \rvert^2 d\bar{x} \frac{d\bar{z}}{\eps^{N-1}} dy_N \\
&\leq \int_{-1}^0 \sum_{k \in I_{W,\eps}} \int_{C^k_\eps} \int_{C^k_\eps} \Bigg \lvert \Bigg(\frac{\partial^{m-1}\varphi}{\partial x_N^{m-1}} (\bar{z},\xi) -  \frac{\partial^{m-1}\varphi}{\partial x_N^{m-1}}(\bar{x},0)  \Bigg) \frac{y_N^{l-1}}{(l-1)!}\\
&\quad + \sum_{s=1}^{l-1} \binom{l-1}{s}  \frac{1}{\eps^{l-1}} \frac{\partial^{m-1} \varphi}{\partial x_N^{m-1}}(\bar{z}, \xi) (\eps y_N)^{l-1-s} (-h_\eps(\bar{z}, \eps y_N))^{s} \Bigg \rvert^2 d\bar{x} \frac{d\bar{z}}{\eps^{N-1}} dy_N \\
\end{split}
\end{equation}
}
and the right-hand side of \eqref{proof: convergence weighted varphi 2} is estimated from above by
{\small
\begin{equation}
\label{proof: convergence weighted varphi 3}
\begin{split}
&\leq C \int_{-1}^0 \sum_{k \in I_{W,\eps}} \int_{C^k_\eps} \int_{C^k_\eps} \Bigg \lvert \frac{\partial^{m-1}\varphi}{\partial x_N^{m-1}} (\bar{z},\xi) -  \frac{\partial^{m-1}\varphi}{\partial x_N^{m-1}}(\bar{z},0)  \Bigg \rvert^2 d\bar{x} \frac{d\bar{z}}{\eps^{N-1}} dy_N\\
&+ C \int_{-1}^0 \sum_{k \in I_{W,\eps}} \int_{C^k_\eps} \int_{C^k_\eps} \Bigg \lvert \frac{\partial^{m-1}\varphi}{\partial x_N^{m-1}} (\bar{z},0) -  \frac{\partial^{m-1}\varphi}{\partial x_N^{m-1}}(\bar{x},0)  \Bigg \rvert^2 d\bar{x} \frac{d\bar{z}}{\eps^{N-1}} dy_N\\
&+C \sum_{s=1}^{l-1} \int_{-1}^0 \sum_{k \in I_{W,\eps}} \int_{C^k_\eps} \int_{C^k_\eps} \Bigg\lvert \frac{\partial^{m-1} \varphi}{\partial x_N^{m-1}}(\bar{z}, \xi)\Bigg\rvert^2 \Bigg\lvert \frac{1}{\eps^{l-1}} (\eps y_N)^{l-1-s} |h_\eps(\bar{z}, \eps y_N)|^{s}   \Bigg\rvert^2   d\bar{x} \frac{d\bar{z}}{\eps^{N-1}} dy_N. \\
\end{split}
\end{equation}
}Now we consider separately the three integrals on the right-hand side of \eqref{proof: convergence weighted varphi 3}. The first integral can be estimated in the following way 
\begin{equation}
\label{FundThmCal}
\begin{split}
&\int_{-1}^0 \sum_{k \in I_{W,\eps}} \int_{C^k_\eps} \int_{C^k_\eps} \Bigg \lvert \frac{\partial^{m-1}\varphi}{\partial x_N^{m-1}} (\bar{z},\xi) -  \frac{\partial^{m-1}\varphi}{\partial x_N^{m-1}}(\bar{z},0)  \Bigg \rvert^2 d\bar{x} \frac{d\bar{z}}{\eps^{N-1}} dy_N\\
&= \int_{-1}^0 \sum_{k \in I_{W,\eps}} \int_{C^k_\eps} \int_{C^k_\eps} \Bigg \lvert \int_0^\xi \frac{\partial^m\varphi}{\partial x_N^m}(\bar{z},t) dt  \Bigg \rvert^2 d\bar{x} \frac{d\bar{z}}{\eps^{N-1}} dy_N \leq C \eps \norma*{\frac{\partial^m\varphi}{\partial x_N^m}}^2_{L^2(W \times (-c\eps,0))},
\end{split}
\end{equation}
Now consider the second integral in \eqref{proof: convergence weighted varphi 3}. We have the following estimate
\begin{equation}
\label{proof: conv varphi II int}
\begin{split}
& \int_{-1}^0 \sum_{k \in I_{W,\eps}} \int_{C^k_\eps} \int_{C^k_\eps} \Bigg \lvert \frac{\partial^{m-1}\varphi}{\partial x_N^{m-1}} (\bar{z},0) -  \frac{\partial^{m-1}\varphi}{\partial x_N^{m-1}}(\bar{x},0)  \Bigg \rvert^2 d\bar{x} \frac{d\bar{z}}{\eps^{N-1}} dy_N\\
&=   \sum_{k \in I_{W,\eps}} \int_{C^k_\eps} \int_{C^k_\eps} \Bigg \lvert \frac{\partial^{m-1}\varphi}{\partial x_N^{m-1}} (\bar{z},0) -  \frac{\partial^{m-1}\varphi}{\partial x_N^{m-1}}(\bar{x},0)  \Bigg \rvert^2 \frac{|\bar{z} - \bar{x}|^{N}}{|\bar{z} - \bar{x}|^{N}} d\bar{x} \frac{d\bar{z}}{\eps^{N-1}} \\
&\leq C  \sum_{k \in I_{W,\eps}} \int_{C^k_\eps} \int_{C^k_\eps}\Bigg \lvert \frac{\frac{\partial^{m-1}\varphi}{\partial x_N^{m-1}} (\bar{z},0) -  \frac{\partial^{m-1}\varphi}{\partial x_N^{m-1}}(\bar{x},0)}{|\bar{z}-\bar{x}|^{N/2}}\Bigg \rvert^2 \eps^N d\bar{x} \frac{d\bar{z}}{\eps^{N-1}}\\
&\leq C \eps \norma*{\frac{\partial^{m-1}\varphi}{\partial x^{m-1}_N}(\bar{x},0)}^2_{B_2^{1/2}(W)} \leq C \eps \norma*{\frac{\partial^{m-1}\varphi}{\partial x^{m-1}_N}(\bar{x},0)}^2_{W^{2,2}(\Omega)} \, ,
\end{split}
\end{equation}
where we have used the classical Trace Theorem and the standard Besov space $B_2^{1/2}(W)$  of exponents  $2, 1/2$. Finally we consider the third integral in \eqref{proof: convergence weighted varphi 3}, which  is easily estimated by using Lemma~\ref{lemma: h_eps} as follows:
{\footnotesize
\begin{equation}
\label{proof: conv varphi III int}
\begin{split}
& \sum_{s=1}^{l-1} \int_{-1}^0 \sum_{k \in I_{W,\eps}} \int_{C^k_\eps} \int_{C^k_\eps} \Bigg\lvert \frac{\partial^{m-1} \varphi}{\partial x_N^{m-1}}(\bar{z}, \xi)\Bigg\rvert^2   \Bigg\lvert \frac{1}{\eps^{l-1}} (\eps y_N)^{l-1-s} |h_\eps(\bar{z}, \eps y_N)|^{s}   \Bigg\rvert^2   d\bar{x} \frac{d\bar{z}}{\eps^{N-1}} dy_N \\
&\leq C \eps^{N-1} \sum_{s=1}^{l-1} \int_{-1}^0 \sum_{k \in I_{W,\eps}} \int_{C^k_\eps} \Bigg\lvert \frac{\partial^{m-1} \varphi}{\partial x_N^{m-1}}(\bar{z}, \xi)\Bigg\rvert^2  \Bigg(\frac{1}{\eps^{l-1}} (\eps)^{l-1-s} |C \eps^{3/2}|^{s}\Bigg)^2 \frac{d\bar{z}}{\eps^{N-1}} dy_N\\
& \leq C \sum_{s=1}^{l-1} \int_{-1}^0 \sum_{k \in I_{W,\eps}} \int_{C^k_\eps} \Bigg\lvert \frac{\partial^{m-1} \varphi}{\partial x_N^{m-1}}(\bar{z}, \xi)\Bigg\rvert^2  \eps^{s}  d\bar{z} dy_N \leq C \eps \norma*{\frac{\partial^{m-1} \varphi}{\partial x_N^{m-1}}}^2_{W^{1,2}(\Omega)}.
\end{split}
\end{equation}}
By using \eqref{FundThmCal}, \eqref{proof: conv varphi II int}, \eqref{proof: conv varphi III int} in \eqref{proof: convergence weighted varphi} we deduce that
{\small
\begin{multline}
\int_{-1}^0 \sum_{k \in I_{W,\eps}} \int_{C^k_\eps} \int_{C^k_\eps}  \Bigg \lvert \frac{1}{\eps^{l-1}}\frac{\partial^{m-1}\varphi}{\partial x_N^{m-1}} (\bar{z},\xi) \frac{(\eps y_N - h_\eps(\bar{z}, \eps y_N))^{l-1}}{(l-1)!}\\
- \frac{y_N^{l-1}}{(l-1)!} \frac{\partial^{m-1}\varphi}{\partial x_N^{m-1}}(\bar{x},0) \Bigg \rvert^2 d\bar{x} \frac{d\bar{z}}{\eps^{N-1}} dy_N \leq C \eps \norma{\varphi}_{W^{m,2}(\Omega)} \to 0,
\end{multline}
}as $\eps \to 0$. This concludes the proof in the case of smooth functions.

Now, if $\varphi \in W^{m,2}(\Omega) \cap W^{m-1,2}_0(\Omega)$, by \cite[Theorem 9, p.77]{Bur_book} there exists a sequence $(\varphi_n)_{n \in \N} \subset  W^{m,2}(\Omega) \cap W^{m-1,2}_0(\Omega) \cap C^\infty(\Omega)$ such that
\[
\varphi_n \to \varphi, \quad\quad \textup{in $W^{m,2}(\Omega_\eps)$},
\]
as $n \to \infty$ hence  ${\rm Tr}_{\partial \Omega} D^\eta \varphi_n = {\rm Tr}_{\partial \Omega} D^\eta \varphi$ for all $|\eta| \leq m-1$. Then
\begin{equation}
\label{proof: approx arg main}
\begin{split}
&\norma*{\frac{1}{\eps^{l-1}} \frac{\partial^{m-l} \varphi}{\partial x_N^{m-l}}(\widehat{\Phi}_\eps(y)) - \frac{y_N^{l-1}}{(l-1)!} \frac{\partial^{m-1}\varphi}{\partial x_N^{m-1}}(\bar{x},0)}_{L^2(\widehat{W}_\eps \times Y \times (-1,0))} \\
&\leq \norma*{\frac{1}{\eps^{l-1}} \frac{\partial^{m-l} \varphi}{\partial x_N^{m-l}}(\widehat{\Phi}_\eps(y)) - \frac{1}{\eps^{l-1}} \frac{\partial^{m-l} \varphi_n}{\partial x_N^{m-l}}(\widehat{\Phi}_\eps(y))}_{L^2(\widehat{W}_\eps \times Y \times (-1,0))} \\
&+ \norma*{\frac{1}{\eps^{l-1}} \frac{\partial^{m-l} \varphi_n}{\partial x_N^{m-l}}(\widehat{\Phi}_\eps(y)) - \frac{y_N^{l-1}}{(l-1)!} \frac{\partial^{m-1}\varphi_n}{\partial x_N^{m-1}}(\bar{x},0)}_{L^2(\widehat{W}_\eps \times Y \times (-1,0))}\\
&+ \norma*{\frac{y_N^{l-1}}{(l-1)!} \frac{\partial^{m-1}\varphi_n}{\partial x_N^{m-1}}(\bar{x},0) - \frac{y_N^{l-1}}{(l-1)!} \frac{\partial^{m-1}\varphi}{\partial x_N^{m-1}}(\bar{x},0)    }_{L^2(\widehat{W}_\eps \times Y \times (-1,0))}\, .
\end{split}
\end{equation}
By using Lemma \ref{lemma: exact int formula}, a Trace Theorem, Poincar\'{e} inequality and a typical diagonal argument, it is not difficult to see that right hand-side of \eqref{proof: approx arg main} tends to zero as $\eps \to 0$, concluding the proof of the first part of the statement.

The second part of the second statement can be proved as follows. By assumption, at least one of the indexes $i_j$ it is different from $N$. This implies that the function $\frac{\partial^{m-l} \varphi}{\partial x_N \partial x_{i_1} \cdots \partial x_{i_{m-l-1}}}$ is not only in $W^{l,2}(\Omega) \cap W_0^{l-1,2}(\Omega)$ but also in $W^{l,2}_{0,W}(\Omega )$. Thus, formula \eqref{eq: exact int formula} and an iterated application of the Poincar\'{e} inequality in the $x_N$ direction, $l-1$ times, yield
\begin{multline*}
\Bigg\lVert \frac{1}{\eps^{l-1}}\frac{\partial^{m-l}\varphi}{\partial x_N\partial x_{i_1} \cdots \partial x_{i_{m-l-1}} }(\hat{\Phi}_\epsilon(y)) \Bigg \rVert_{L^2(W \times Y \times (-1,0)} \leq C \Bigg\lVert \frac{\partial^{m-1}\varphi}{\partial x^{l}_N \partial x_{i_1} \cdots \partial x_{i_{m-l-1}} }(\hat{\Phi}_\epsilon(y)) \Bigg \rVert_{L^2(W \times Y \times (-1,0)}\,
\end{multline*}
which allows to conclude since the right-hand side of the previous inequality tends to zero as $\eps \to 0$ in virtue of Lemma \ref{lemma: unfolding convergence poly}(ii) and of the vanishing of the trace of $\frac{\partial^{m-1} \varphi}{\partial x_N^l \partial x_{i_1} \cdots \partial x_{i_{m-l-1}}}$ on $W$.

\end{proof}


\subsection*{Acknowledgment}
The authors are deeply indebted to Prof. J.M. Arrieta for valuable suggestions and discussions. The first author gratefully acknowledges the support of the \emph{Swiss National Science Foundation}, SNF, through the Grant No.169104. The second author acknowledges financial support from  the INDAM - GNAMPA project 2017 ``Equazioni alle derivate parziali non lineari e disuguaglianze funzionali: aspetti geometrici ed analitici" and the INDAM - GNAMPA project 2019 ``Analisi  spettrale  per  operatori  ellittici  con  condizioni  di  Steklov  o  parzialmente incernierate''. The  authors are also members of the Gruppo Nazionale per l'Analisi Matematica, la Probabilit\`{a} e le loro Applicazioni (GNAMPA) of the Istituto Nazionale di Alta Matematica (INdAM).

\end{document}